\documentclass[11pt]{article}
\usepackage{bbm}
\usepackage{amsfonts}
\usepackage{amssymb}
\usepackage{stmaryrd}
\usepackage{bm}
\usepackage{bbding}
\usepackage{amsmath}
\usepackage[colorlinks,
            linkcolor=blue,
            anchorcolor=blue,
            citecolor=blue
            ]{hyperref}
 \usepackage{latexsym,amssymb,mathrsfs,fancyhdr}
\font\tenmsb=msbm10 \font\sevenmsb=msbm7 \font\fivemsb=msbm5
\catcode`\@=11 \ifx\amstexloaded@\relax\catcode`\@=\active
\endinput\else\let\amstexloaded@\relax\fi
\def\spaces@{\space\space\space\space\space}
\def\spaces@@{\spaces@\spaces@\spaces@\spaces@\spaces@}
\def\relaxnext@{\let\next\relax}
\def\accentfam@{7}

\def\noaccents@{\def\accentfam@{0}}
\def\mathcal{\relaxnext@\ifmmode\let\next\mathcal@\else
\def\next{\Err@{Use
\string\mathcal\space only in math mode}}\fi\next}
\def\mathcal@#1{{\mathcal@@{#1}}} \def\mathcal@@#1{\noaccents@\fam\tw@#1}

\def\1{{\frac12}}

    \def\eqref #1{{(\ref{#1})}}
\def\Bbb{\relaxnext@\ifmmode\let\next\Bbb@\else
\def\next{\Err@{Use \string\Bbb\space only in math mode}}\fi\next}

\def\Bbb@#1{{\Bbb@@{#1}}}
\def\Bbb@@#1{\noaccents@\fam\msbfam#1}
\newfam\msbfam \textfont\msbfam=\tenmsb \scriptfont\msbfam=\sevenmsb
\scriptscriptfont\msbfam=\fivemsb
\def\N{{\mathbb N}} \def\Z{{\mathbb Z}}   
\def\R{{\mathbb R}} \def\T{{\mathbb T}}  
\def\C{{\mathbb C}}

\newtheorem{Theorem}{Theorem}
\newtheorem{Lemma}{Lemma}[section]

\newtheorem{Proposition}{Proposition}[section]

\newtheorem{Remark}{Remark}[section]

\newtheorem{definition}{Definition}[section]

\@addtoreset{equation}{section}

\newcommand{\qed}{\nolinebreak\hfill\rule{2mm}{2mm}
\par\medbreak}

\newcommand{\la}{\langle } \newcommand{\ra}{\rangle }

\newcommand{\beq}{\begin{equation} } \newcommand{\eeq}{\end{equation} }

\begin{document}

\setlength{\columnsep}{5pt}
\title{Quasi-periodic  solutions of NLS with Liouvillean Frequency}

\author{Xindong  Xu$^a$, Jiangong You$^b$, Qi Zhou$^c$\\
{\footnotesize a: School of Mathematics, Southeast University }\\
{\footnotesize  Nanjing 210096, P.R.China}\\
{\footnotesize b: Chern Institute of Mathematics and LPMC, Nankai University }\\
{\footnotesize  Tianjin 300071, P.R.China}\\{\footnotesize c: Department of Mathematics, Nanjing University, }\\
{\footnotesize Nanjing 210093, P.R.China}\\
 {\footnotesize Email: xindong.xu@seu.edu.cn, jyou@nankai.edu.cn, qizhou@nju.edu.cn  }}

\date{}
\maketitle

\noindent{\bf Abstract:}\\

Quasi-periodic  solutions  with Liouvillean frequency of forced nonlinear  Schr\"odinger equation are constructed. This is based on an infinite dimensional KAM theory for Liouvillean frequency.
\\

\noindent{\bf R\'esum\'e:}\\

Les solutions quasi-p\'eriodiques avec les fr\'equences de Liouville de l'\'equation de Schr\"odinger non lin\'eaire forc\'ee sont construites.  C'est fond\'e sur la th\'eorie KAM de dimension infinie pour la fr\'equence de Liouville.

\section{Introduction}

In 1989's, Kuksin  \cite{ku} first constructed quasi-periodic solutions for $1$d NLS  equation with  Dirichlet boundary conditions
 by infinite dimensional KAM theory. Following \cite{ku}, mathematicians study Hamiltonian PDEs (such  as wave equation, KDV and etc) with periodic boundary condition or in higher space dimension,  many  other methods are developed. They also consider Hamiltonian PDEs with derivative or finitely differentiable nonlinearities. For more details, one may refer  to \cite{BB1,Bourgain3,EK,GY2,GXY,LiuYuan,KPo,KLiang,Ku2, Wayne} and the references therein.

Note that all the  quasi-periodic solution constructed above must satisfy  some Diophantine condition. This is  the key observation by Kolmogorov in 1954. We recall a vector $\omega\in\R^d$ is said to be  {\it Diophantine} if $$|\la k,\omega\ra|\geq {\gamma\over |k|^\tau},\quad k\in \Z^d\backslash\{0\}.$$ 
Later, people find results which works for Diophantine condition can be parallelly generalized to Brjuno condition, which is 
$$\mathcal{B}(\omega):= \sum_{n\geq 0} \frac{1}{2^n} \max_{0<\|k\|\leq 2^n,k\in\Z^d} \ln \frac{1}{ |\la k,\omega\ra| }<\infty.$$ 
If $\omega$ is not Brjuno,  we call it is  Liouvillean.    The question is that whether it is possible to obtain some quasi-periodic solution with Liouvillean frequency?  In this paper, we will establish the existence of quasi-periodic solution beyond Brjuno frequency.  Before introducing the precise result, we need to give some necessary definitions.

For $\bar \omega=({\bar\omega_1},\bar\omega_2)$ with $\bar\omega_1=(\alpha,1)$, $\alpha\in \R\backslash \mathbb Q$, $\bar \omega_2\in \R^d$,  we  say that the frequency $\bar \omega$ is  \textit{weak Liouvillean}, if there exist $\gamma>0$ and $\tau>d+6$, such that
\begin{eqnarray}\label{frequency assumption}\left\{
\begin{array}{ll}
 \beta(\alpha):=\limsup\limits_{n>0} \frac{\ln\ln q_{n+1}}{\ln q_n}<\infty, &\\
\label{wl} |\la k,\bar \omega_1\ra+\la l,\bar \omega_2\ra|\geq {\gamma\over (|k|+|l|)^\tau},  &{\rm for }\,   k\in \Z^2,l\in \Z^d\backslash \{0\}.
\end{array}\right.
\end{eqnarray}
where $\frac{p_n}{q_n}$  is  the continued fraction approximates to $\alpha$ (c.f. section \ref{sec:2.1}).  Denote by $WL(\gamma, \tau,\beta)$ the set of such frequency  and by $WL$ the union
\[WL=\bigcup_{\gamma>0,\tau>d+6,\beta<\infty}WL(\gamma, \tau, \beta). \]
 It is obvious that  $WL$ is of full Lebesgue
measure, and if $\bar\omega\in WL$, then it is not necessarily  to be  Brjuno. \footnote{In the case $d=2$,  if $\mathcal{B}(\bar \omega_1)<\infty$, then $\beta(\alpha)=0$.}

While our method works for other Hamiltonian PDEs, as an example, we  study the quasi-periodic solution of  forced  NLS:
\begin{equation}\label{equ}
 iu_t- u_{xx}+v(x) u+\epsilon f(\omega t,x,u,\bar u;\xi)=0%{\partial\over \partial\bar u}
\end{equation} on segment $[0,\pi]$ with Dirichlet boundary condition
$$u(t,0)=0=u(t,\pi),-\infty <t<+\infty.$$ Our main result is the following:

\begin{Theorem}\label{thm-appli-1}
Let $\bar \omega\in WL$,  $\xi\in\mathcal O=({1\over2},{3\over2})$.
 The function $v(x)$ is real analytic with $\int_0^\pi |v(x)|dx<1$,
 $f(\Theta,x,u,\bar u;\xi)$ is assumed to be real analytic on $\Theta,x,u,\bar u$ and Lipschitz on $\xi$.    Then for
any small $\gamma>0$, there exists $\epsilon_0>0$ and $\mathcal O_\gamma\subset \mathcal O$ with $| \mathcal O \backslash\mathcal O_\gamma | =O(\gamma) $,  such that equation $(\ref{equ})$ has a $C^\infty$ smooth  quasi-periodic  solution with frequency $\omega=\xi \bar{\omega}$  for any $\xi\in\mathcal O_\gamma$ if $\epsilon<\epsilon_0$.
\end{Theorem}

Before giving its proof, let us make some comments on the result.

 We choose  NLS as a model mainly because it is one of the most important equation in mathematical physics,  many questions are still open.  It has been a long time for people  to  construct  quasi-periodic solution of NLS by KAM,  real breakthrough was  recently made by   Elliasson-Kuksin \cite{EK},  who  established quasi-periodic solution of   NLS with $x\in \T^d$. We should mention that the existence of quasi-periodic with Diophantine frequency for NLS in higher dimension was first proved by Bourgain \cite{Bourgain3,Bourgain4} by CWB method.  To  ensure localization properties of the eigenfunctions, $v(x)$  from  the operator $\partial_{xx}+v(x)$ is usually substituted by a \textquotedblleft convolution potential\textquotedblright  (see \cite{EK,GY2,P3}), which will provide parameters required by KAM theorem.  However,  in our result,  the potential $v(x)$ serves as a  multiplicative operator, the result holds for any fixed multiplicative operator $v(x)$,  we do not extract parameters
from $v(x)$, the role of parameter is being played by $\xi$ from $\omega=\xi({\bar\omega_1},\bar\omega_2)$. In fact, the existence of this  kind of solution (quasi-periodic solution with frequency vary in a line) was first proposed by
by Bourgain \cite{Bourgain1} and Eliasson \cite{E2}. In the infinitely dimensional Hamiltonian setting,  this has
been first proved by Geng-Ren \cite{GR} for 1-dimensional wave equation and then Berti-Biasco \cite{BeBi} for $1$-dimension NLS.  Berti-Bolle \cite{BB1}  answered this question for the forced
  NLS like \eqref{equ} with  differentiable nonlinearity and $x\in \T^d$.
Comparing with \cite{BB1}, Berti-Bolle relax the perturbation to be finite differentiable  and forced  by {\it Diophantine} frequency, while  the system we consider is  forced by {\it Liouvillean} frequency and the perturbation is  analytic.

%\begin{Remark}
%Existence of quasi periodic with Liouvillean frequency can also be proved for wave equation, beam equation and etc.
%\end{Remark}

As we should mentioned, our work is also motivated by Avila-Fayad-Krikorian,  Hou-You's recent work \cite{AFK,HoY},  where
they consider rotation  reducibility result  of quasi-periodic $SL(2;\R)$ cocycles with Liouvillean frequency. Note quasi-periodic $SL(2;\R)$ cocycles can be viewed as two dimensional {\it linear} Hamiltonian,  reducibility of  $SL(2;\R)$ cocycles is equivalent to the quasi-periodic solution of the corresponding Hamiltonian systems. Readers can refer \cite{A, FK,ZW} for related results. We just emphasize that  reducibility of quasi-periodic $SL(2,\R)$ cocycles with Liouvillean frequency is quite meaningful, since the dynamics of quasi-periodic $SL(2,\R)$ cocycles are  closely  related to the spectral theory of one-dimensional quasi-periodic  Schr\"odigner operators, the reducibility of $SL(2;\R)$ cocycles with  Liouvillean frequency plays a quite important role in recent advances of spectral theory of  quasi-periodic  Schr\"odinger operators, for example, Avila's global theory of one-frequency quasi-periodic  Schr\"odigner operators \cite{A,Aglobal}, the solution of Aubry-Andr\'{e}-Jitomirskaya's conjecture \cite{AYZ1}.
We further mention that before our work, Wang-You-Zhou \cite{WYZ} already generalized Avila-Fayad-Krikorian's result \cite{AFK} to finite dimensional nonlinear Hamiltonian system, where they obtained response solution of harmonic oscillators, however, as we will discuss in section \ref{homoeq}, the key techniques are quite different compared to this paper.

Finally, let's comment on the innovations of our results.  The proof of the theorem is  based on infinite dimensional KAM theory, it is  well-known that  the key point of KAM theory is the solution of homological equation. The typical homological equation we  meet can be written as
\begin{eqnarray}\label{kuksinlemma}
-{\rm i}\partial_{\omega}u+\zeta u+b(x)u=f(x),\qquad x\in \T^n.
\end{eqnarray}
In fact, this kind of equation was already met when Kuksin \cite{Ku2}  studied KDV equations(also \cite{KPo}),  and also by Liu-Yuan \cite{LiuYuan} when they  studied one-dimensional derivative NLS. We will provide a quite general method for solving this kind of equation, and is believed to
  have further applications. Compared  to \cite{LiuYuan,KPo,Ku2},  the method is totally different and
which even works for Liouvillean frequency (not merely Diophantine frequency as in \cite{LiuYuan,KPo,Ku2}), this is one novelty
 of the paper. Readers are invited to consult section \ref{homoeq} for more discussions.

 We emphasize that in all the results mentioned above \cite{A,Aglobal,AFK,AYZ1,HoY, WYZ,ZW}, the frequency is one frequency (thus two frequencies in the continuous case), however, our method works for multifrequency.
To the best knowledge of the authors, our result  gives  the first result  regarding on the quasi-periodic solutions with Liouvillean frequency for Hamiltonian PDE, and it also gives the  first positive result regarding on multifrequency Liouvillean frequency (even for the linear finite dimensional  Hamiltonian case)! One can not hope our result works for  {\it any} Liouvillean frequency, since in the linear cocycle case, Avila and Jitomirskaya \cite{AJ} already  proved that
 there exists two dimensional frequency, such that for  typical  analytic potential, the corresponding Schr\"odinger cocycle has positive Lyapunov exponent for almost every energies. Thus the corresponding Hamiltonian system doesn't exist  quasi-periodic solution.

%Finally, let's  comment on the key difficulty of the paper.

\section{Preliminaries}
\subsection{Continued fraction expansion.}
\label{sec:2.1}
Let $\alpha \in \R\backslash \mathbb Q$ be irrational.
We first set $$a_0=0,\alpha_0=\alpha$$
and then we define inductively for $n\geq 1:$
$$a_n=[\alpha_{n-1}^{-1}],\alpha_n=\alpha_{n-1}^{-1}-a_n:=G(\alpha_{n-1})=\{{1\over\alpha_{n-1}}\}.$$
We define
\begin{eqnarray*}
p_0=0, \quad  q_1=a_1, \quad q_0=1,\quad p_1=1
\end{eqnarray*}
and
\begin{eqnarray*}
p_n=a_np_{n-1}+p_{n-2},\quad
q_n=a_nq_{n-1}+q_{n-2}.
\end{eqnarray*}
Then $(q_n)$ is the sequence of denominators of the best rational
approximations for $\alpha\in \R\backslash \mathbb Q$.  To be more precise,   for any $ 1\leq k<q_n$, one has
\begin{equation}\label{lowerbound1} \|k\alpha\|\geq\|q_{n-1}\alpha\|,\quad {1\over q_{n+1}+q_n}\leq \|q_n\alpha\|\leq {1\over q_{n+1}},\end{equation}
where $\|x\|=\inf\limits_{p\in\Z}|x-p|$.

\subsection{CD-Bridge}For any $\alpha\in\R\backslash\mathbb Q$, let $(q_n)$ be the sequence of denominators of best rational approximations.
We choose two particular subsequences of $(q_n)$, the first is  $(q_{n_k}) $ which we denote  by  $(Q_k)$ for simple, the second is  $(q_{n_k+1})$  which we denote by $(\overline Q_k)$. The properties required
from our choice of the subsequence $(Q_k)$ are summarized below.   The definition of CD-Bridge is required.

\begin{definition}{ \cite{AFK}}
Let $0< \mathcal A \leq \mathcal B\leq\mathcal C$, We say that the pair of denominators $(q_\ell,q_n)$ forms a CD-bridge if :
\begin{enumerate}
\item  $q_{i+1}\leq q_i^{\mathcal A},\quad i=\ell,\cdots,n-1$
\item  $ q_\ell^{\mathcal B}\leq q_n \leq q_\ell^{\mathcal C}$.
\end{enumerate}
\end{definition}

\begin{Lemma}\label{CDbridge}{ \cite{AFK}} For any $\mathcal A>0$, there exists a subsequence $(Q_k)$, such that $Q_0=1$ and  for each $k\geq 0$,
$Q_{k+1}\leq {\overline  Q_{k}^{\mathcal A^4}}$, and either ${\overline  Q}_{k}\geq Q_k^{\mathcal A}$ or the pairs $({\overline  Q}_{k-1},Q_k)$ and $(Q_k,Q_{k+1})$ are both $CD(\mathcal A,\mathcal A,\mathcal A^3)$ bridges.
\end{Lemma}

% We say $\alpha$ is not super-liouvillean  if $\tilde U<\infty$, where
%\[\tilde U(\alpha)=\sup\limits_{n>0}{\ln\ln q_{n+1}\over \ln q_n}.\]
In the sequel, we assume $\mathcal A\geq 10$, and $(Q_n)$ is the selected subsequence as in
Lemma \ref{CDbridge}. Note if   $\beta(\alpha)=\limsup_{n>0} \frac{\ln\ln q_{n+1}}{\ln q_n}<\infty$, then  $\widetilde{U}(\alpha):=\sup_{n>0} \frac{\ln\ln q_{n+1}}{\ln q_n}<\infty$. Then we have the following:
% Then one has following result at once,  one can refer \cite{KWZY} for this easy proof.
\begin{Lemma}{ \cite{KWZY}}\label{sup-liouvillean}
If $\tilde U(\alpha)<\infty$,  then  there is $Q_n\geq Q^{\mathcal A}_{n-1}$ for any $ n\geq 1$. Furthermore, one has
$$\sup\limits_{n>0}{\frac{\ln\ln Q_{n+1}}{\ln Q_n}}\leq U(\alpha),\quad \ln Q_{n+1}\leq {Q_n^U},$$
where $U(\alpha)=\tilde U(\alpha)+4{\ln{\mathcal A}\over\ln 2}$.
\end{Lemma}

\section{An Infinite Dimensional KAM Theorem}

The main result will be proved by a generalized KAM theorem for Liouvillean frequency. In this section, we introduce this basic KAM result.

We start by introducing the notations.  The Lipschitz norm of  a function $f(\xi)$ with $\xi\in\mathcal O\subset\R$  is defined as   $$|f(\xi)|_{\mathcal O}^*= |f(\xi)|_{\mathcal O} +|f(\xi)|_{\mathcal O}^{\mathcal L},$$ where
$|f(\xi)|_{\mathcal O}= \sup\limits_{\xi\in \mathcal O}|f(\xi)|, |f(\xi)|_{\mathcal O}^{\mathcal L}=\sup\limits_{\xi,\eta\in \mathcal O, \xi\neq\eta}{|f(\xi)-f(\eta)|\over |\xi-\eta|}.$
Let $D_r =\{(\theta,\varphi)\in \T^2\times \T^d, |\Im \theta|+|\Im \varphi|<r\}$. For a bounded holomorphic (possibly with parameter) function $g(\theta,\varphi;\xi)=\sum\limits_{(k,l)\in \Z^2\times \Z^d}\hat g_{(k,l)}(\xi)e^{\rm i\la (k,l),(\theta,\varphi)\ra}$ on $D_r$, we let
$$|g|_{r,\mathcal O}^*=\sum_{|k|+|l|\in\Z^{2+d}} |\hat g_{(k,l)}(\xi)|_{\mathcal O}^*e^{(|k|+|l|)r}.$$
We denote by $\mathscr B_r(\mathcal O)$ the set of these functions, and
 for any $K\in \Z^{+},$ we define the  truncation operator $\mathcal T_K$  as
 \begin{equation}\mathcal T_K g(\theta,\varphi;\xi)=\sum\limits_{ |k|+|l|<K}\hat g_{(k,l)}(\xi)e^{\rm i\la (k,l),(\theta,\varphi)\ra},\end{equation}
also denote \begin{equation}\label{average}[g(\theta,\varphi;\xi)]_{\varphi}=\int_{\T^d}g(\theta,\varphi)d\varphi,\,[g(\theta,\varphi;\xi)]=\int_{\T^2\times \T^d}g(\theta,\varphi)d\theta d\varphi.\end{equation}

Let $\ell_{\C}^{a,\rho}$  be the Hilbert space of sequence $z=(z_1,z_2,\cdots)$ with
$$|z|_{a,\rho}^2 =\sum_{{p\geq 1}}|z_p|^2 p^{2\rho} e^{2a|p|}<\infty,$$
where $a>0$ and $\rho>0$. For  $r,s>0$, we then introduce the complex neighborhoods of  $ {\T^{2+d}}\times \{0,0,0\}$
by \begin{eqnarray*}D(r,s)&=&\{({\Theta},\mathcal I,z,\bar z):|{\rm Im} \Theta|<r, |\mathcal I|<s^2, {|z|}_{a,\rho}<s,
{|\bar {z}|}_{a,\rho}<s\}\\
&\subseteq& {\C^{2+d}}\times \C^{2+d}\times\ell_{\C}^{a,\rho}\times\ell_{\C}^{a,\rho}:=\mathscr{P}^{a,\rho}_{\C},\end{eqnarray*} where $\T^{2+d}$ is the usual ${2+d}$-torus, $|\cdot|$ denotes the sup-norm of complex vectors for
\begin{equation}\Theta=(\theta,\varphi)\in \T^2\times \T^d,\mathcal I=(I,J)\in \R^2\times \R^d.\end{equation}
 For any  $W=(X,Y,U,V)\in\mathscr{P}^{a,\rho}_{\C}$, the  weighted phase norm is defined to be
\begin{equation}\label{norm}|W|_{s}=:|W|_{s,a,\rho}=|X|+{1\over s^2}|Y|+{1\over s}|U|_{a,\rho}+{1\over s}|V|_{a,\rho}.\end{equation}
For any map ${\mathcal W}:D(r,s)\times \mathcal O\rightarrow  \mathscr{P}^{a,\rho}_{\C}$,  we define its norm as
$$|\mathcal W|_{s,D(r,s)\times \mathcal O}=\sup\limits_{D(r,s)\times \mathcal O}|\mathcal W|_{s},$$
$$|\mathcal W|_{s,D(r,s)\times \mathcal O}^{\mathcal L}=\sup\limits_{\xi,\eta\in \mathcal O, \xi\neq\eta}{|\triangle_{\xi\eta} \mathcal W|_{s,D(r,s)\times \mathcal O}\over |\xi-\eta|},$$
where $\triangle_{\xi\eta} {\mathcal W}={\mathcal W}(\cdot,\xi)-{\mathcal W}(\cdot,\eta)$ and the supremum is taken over $\mathcal O$.

We also need  the operator norm  $\pmb{\pmb |} \cdot\pmb{\pmb |}_{s,\tilde s}$ below,
$$\pmb{\pmb |}A\pmb{\pmb |}_{s,\tilde s}=\sup\limits_{W\neq0}{|AW|_{s}\over|W|_{\tilde s}},$$
where $|\cdot|_s$ is the shorten of $|\cdot|_{s,a,\rho}$ defined in \eqref{norm}, and $|\cdot|_{\tilde s}$ defined similarly. For $s\geq \tilde s$, these norms satisfy $\pmb{\pmb |}AB\pmb{\pmb |}_{s,\tilde s}\leq \pmb{\pmb |}A\pmb{\pmb |}_{s,s}\cdot\pmb{\pmb |}B\pmb{\pmb |}_{\tilde s,\tilde s} $ since $|W|_s\leq |W|_{\tilde s}$.

If the function $F$ is analytic in space coordinate, we usually take Taylor--Fourier series as:
\beq\label{2.2}
 F(\Theta, \mathcal I, z, \bar z;\xi )=\sum_{\iota,\mu,\alpha,\beta }
F_{\iota \mu \alpha\beta }(\xi)\mathcal I^\iota e^{{\rm i}
 \la \mu,\Theta\ra}z^{\alpha} \bar z^{\beta },\eeq
 where the coefficient functions $F_{\iota \mu \alpha\beta }(\xi)$ are Lipschitz on $\xi$, the vectors $\alpha\equiv (\cdots,\alpha_n,\cdots)_{n\geq 1}$, $\beta \equiv
(\cdots, \beta _n, \cdots)_{n\geq 1}$  have finitely many non-zero  components $\alpha_n,\beta_n\in \N$, $z^{\alpha} \bar z^{\beta }$ denotes $\prod_n
z_n^{\alpha_n}\bar z_n^{\beta_n}$ and finally $\la\cdot,\cdot\ra$ is the standard inner product in $\C^d$.

In this paper, we will consider the perturbed Hamiltonian on $ D(r,s) \times  \mathcal O$,
\begin{equation}\label{ham}H=\la\omega_1,I\ra+\la\omega_2,J\ra+\sum_{p\geq 1}\mathbf \Omega_p(\xi) |z_p|^2+P(\theta,\varphi,z,\bar z;\xi).\end{equation}
endowed with  the symplectic structure
%\sum\limits_{\nu=1}^{2+d}d\mathcal I_\nu\wedge d\Theta_\nu+i\sum\limits_{p\geq 1}dz_p\wedge d\bar z_p=
 $$dI\wedge d\theta+d J\wedge d\varphi+i\sum_{p\geq 1}dz_p\wedge d\bar z_p.$$
 The perturbation $P(\theta,\varphi,z,\bar z;\xi)$ is real analytic in space coordinates $\theta,\varphi,z,\bar z$ and Lipschitz in parameters $\xi$.  For each $\xi\in \mathcal O$, the Hamiltonian  vector field $X_P=(-P_{(\theta,\varphi)},0,iP_z,-iP_{\bar z})$ defines  a real analytic map $X_P:\mathscr{P}^{a,\rho}_{\C}\rightarrow \mathscr{P}^{a,\rho}_{\C}$ near $\T^{2+d}\times \{0,0,0\}$. We denote the weighted norm of  $X_P$ to be
 $$|X_P|^*_{s,D(r,s)\times \mathcal O}=|X_P|^{\mathcal L}_{s,D(r,s)\times \mathcal O}+|X_P|_{s,D(r,s)\times \mathcal O}.$$  Then we have the following infinite dimensional KAM theorem:
\begin{Theorem}\label{kam} For any given $\beta<\infty,\tau>d+6,s>0,r>0,\gamma>0$,   let  $\omega=\xi \bar{\omega}$ where $\bar{\omega} \in WL(\gamma,\tau,\beta)$,   $\xi\in\mathcal O=({1\over2},{3\over2})$.
Suppose the  Hamiltonian $\eqref{ham}$ satisfy
 $$\sup_{p\geq1}|\mathbf\Omega_p-p^2|^*_{\mathcal O}<{1\over 2}.$$
\noindent Then  there exists $\epsilon_0(\tau,\beta,\gamma,s,r)>0$,  such that for any real analytic perturbation $P(\theta,\varphi,z,\bar z;\xi)$  with $$\epsilon=|X_P|^{*}_{s,D(r,s)\times \mathcal O}\leq\epsilon_0,$$
  there exists a Cantor set $\mathcal O_\gamma$ of  $\mathcal O$ with ${\rm
 meas}(\mathcal O\setminus \mathcal O_\gamma)=O(\gamma)$ and a Lipschitz family  symplectic  map $\Phi:\T^{2+d}\times \mathcal O_
 \gamma\rightarrow \mathscr{P}^{a,\rho}_{\C}$ which is  $C^\infty$ smooth in $\theta,\varphi$, such that
 $\eqref{ham}$  is  transformed to
 $$H^*=e^*(\theta;\xi)+\la\omega_1,I\ra+\la\omega_2,J\ra+\sum\limits_{p\geq1} (\mathbf\Omega_p^*(\xi)+B^*_p(\theta;\xi)) |z_j|^2+P^*(\theta,\varphi,z,\bar z;\xi),$$
where   $P^*(\theta,\varphi,z,\bar z;\xi)=\sum\limits_{|\alpha+\beta|\geq 3}P^*_{\alpha,\beta}(\theta,\varphi;\xi)z^\alpha\bar z^{\beta}.$
\end{Theorem}

\begin{Remark}
We emphasize that the perturbation is independent of the  action variable $I$ and $J$, this fact is crucial for our results.
\end{Remark}
\subsection{Main ideas of the proof}

Theorem\textbf{ \ref{kam}} is proved by modified  KAM  theory which involves an infinite sequence of change of variables. The philosophy of KAM theory is to construct a series of symplectic transformation which makes the perturbation smaller and smaller at the cost of excluding a small set of parameters.  
Compared to the classical KAM scheme, due to the Liouvillean property of $\bar{\omega}_1$ by condition \eqref{wl}, some $\theta$ dependent terms have to be preserved as a normal form  under KAM iteration.  Thus we  have a generalized  Hamiltonian 
\begin{equation}\label{3.7}
H_{n}=e(\theta;\xi)+\la\omega_1,I\ra+\la\omega_2,J\ra+\sum\limits_{p\geq1} (\mathbf\Omega_p^
 n(\xi)+B_p^n(\theta;\xi)) |z_p|^2+P_n(\theta,\varphi,z,\bar z;\xi).
 \end{equation}
 where $B_p^n(\theta;\xi)$ is of size $\epsilon_0$, and the perturbation 
$P_n(\theta,\varphi,z,\bar z;\xi)$ is of size $\epsilon_n$.  In the following, we will construct a symplectic transformation $\Phi_{n+1}$ which is close to the identity (Proposition \ref{KAM iteration}), such that  $\Phi_{n+1}$ transform $(\ref{3.7})$ to 
\begin{eqnarray*}
 H_{n+1} &=& e(\theta;\xi)+\la\omega_1,I\ra+\la\omega_2,J\ra\\ &+& \sum\limits_{p\geq1} (\mathbf\Omega_p^
 {n+1}(\xi) +B_p^{n+1}(\theta;\xi)) |z_p|^2+P_{n+1}(\theta,\varphi,z,\bar z;\xi),
 \end{eqnarray*}
where $B_p^{n+1}(\theta;\xi)$ is still  of size $\epsilon_0$, and the perturbation 
$P_{n+1}(\theta,\varphi,z,\bar z;\xi)$ is of size $\epsilon_{n+1}$.  However, compared to classical KAM iteration, $\epsilon_n$ shrinks to $0$ much faster (other than $\epsilon_{n+1}=\epsilon_n^{3/2}$), and Proposition  \ref{KAM iteration} is proved with finite KAM iteration steps. The reason is the following:   to  eliminate the effect taken by  $B_p^n(\theta;\xi)$, when one solves the homological equation (Proposition \ref{basic lemma}), one has to shrink the analytical strip of $\theta$ very quickly (that's reason why we can only  obtain $C^\infty$ soluton), as a consequence, $\epsilon_n$ has to shrink much faster otherwise the homological equation doesn't admit any analytical solution. Finally,  finite KAM iteration steps are needed to ensure the fast decay of $\epsilon_n$.

\subsection{The infinite induction}

To begin with  iteration, we first fix   $ \epsilon_0, r_0>0,s_0>0, \tau>d+6,\mathcal A>10$ and $\alpha\in\R\setminus\mathbb Q$ with $U(\alpha)<\infty$, let  $(Q_n)$  be the selected subsequence as in
Lemma \ref{CDbridge}. We then  define  the iteration sequences for $n\geq 1$:
\begin{equation}\label{infinite iteration sequence}
\begin{array}{ll}
s_{n}= \epsilon_{n-1}^{({4\over3})^{2+[{2^{n+1}c{\tau}U\ln Q_{n+1}\over 6\tau+9}]}-1}\cdot s_{n-1},& r_{n}={r_0\over4Q^{4}_{n}},\\
\epsilon_{n}=\epsilon_{n-1}\cdot Q_{n+1}^{-2^{n+1}c\tau U},&\mathcal E_n=\sum\limits_{m=0}^{n-1}\epsilon_m^{1\over2},\\
\gamma_n=\gamma_0-3\sum\limits_{m=0}^{n-1}{\epsilon_m^{1\over2}},&K_n=r_0^{-1}40Q_{n+1}^{4}\ln\epsilon_n,\\
D_n=D( r_n, s_n).&
\end{array}
\end{equation}
where $c$ is a  global constant with $c > {18\tau+27\over 2\tau U}$.

 For convenience, for $r,s,M_1,M_2,m>0$ and the parameter set $\mathcal O$, we define the space $\mathcal F_{r,s,\mathcal O}(M_1, M_2,m)$ to be the functions
\begin{eqnarray*}e(\theta;\xi)+\sum\limits_{p=1}^\infty (\it\Omega_p(\xi)+B_p(\theta;\xi)) |z_p|^2+P(\theta,\varphi,z,\bar z;\xi)\end{eqnarray*}
which satisfy
\begin{equation*}
\left\{
\begin{array}{ll}
|e(\theta;\xi)|^*_{r,{\mathcal  O}}\leq {M}_1, &|{\it\Omega}_p(\xi)|^*_{\mathcal O}\leq {1\over2}+M_1, \\
 |B_p(\theta;\xi)|^*_{r,\mathcal O}\leq M_2,&|X_{P}|_{s,D(r,s)\times \mathcal O}^*\leq  m.
 \end{array}\right.
\end{equation*}

Now we have the following result:
%\begin{equation}
%\left\{
%\begin{array}{ll}
%|e_n|^*_{r_{j+1},{\mathcal  O}_{j+1}}\leq \mathcal E_n, &|X_{\sum\limits_{p=1}^\infty \it\Omega_p^ n(\xi)|z_p|^2|}|_{r_n,s_n}^*\leq \mathcal E_n, \\
% |X_{\sum\limits_{p=1}^\infty B_p^
% n(\xi)|z_p|^2|}|_{r_n,s_n}^*\leq \mathcal E_n,&|X_{P_n}|_{r_n,s_n}^*\leq \epsilon_n.
% \end{array}\right.
%\end{equation}

\begin{Proposition}\label{KAM iteration}
Suppose that  $\epsilon_0$ is small enough so that
\begin{eqnarray}\label{parametersetting}\epsilon_0 &\leq& \min\{\frac{(r_0s_0\gamma_0)^{12\tau+36}}{Q_1^{2c\tau U}}, e^{-2c\tau U}\},\\
  \ln\epsilon_0^{-1}&\leq&\epsilon_0^{-{1\over12\tau+18}}.\end{eqnarray}
Then the following holds for all $n>0$:
 Let
\begin{equation}\label{hamilton}H_{n}=e_n(\theta;\xi)+\la\omega_1,I\ra+\la\omega_2,J\ra+\sum\limits_{p\geq1}(\mathbf\Omega_p^
 n(\xi)+B_p^n(\theta;\xi)) |z_p|^2+P_n(\theta,\varphi,z,\bar z;\xi)\end{equation}
 which satisfy
 \begin{enumerate}
 \item For  parameter $\xi\in \mathcal O_n$, $(k,l)\in\Z^2\times\Z^d$ with $|k|+|l|\leq K_{n}$ and $p,q\geq 1$ there is
\begin{eqnarray}\label{newstart 2}
&&|\la k,\omega_1\ra+\la l,\omega_2\ra\pm{\mathbf\Omega}^n_p|\geq  {\gamma_n\over (|k|+|l|+1)^\tau},\nonumber\\
&&|\la k,\omega_1\ra+\la l,\omega_2\ra\pm({\mathbf\Omega}^n_p+{\mathbf\Omega}^n_q)|\geq  {\gamma_n\over (|k|+|l|+1)^\tau},\\
&&|\la k,\omega_1\ra+\la l,\omega_2\ra\pm({\mathbf\Omega}^n_p-{\mathbf\Omega}^n_q)|\geq  {\gamma_n\over (|k|+|l|+1)^\tau},\,\,|l|+|p-q| \neq 0;\nonumber
\end{eqnarray}
 \item  The functions $e_n(\theta;\xi),B^n_p(\theta;\xi)\in \mathscr B_{r_n}(\mathcal O_n)$ with $[B^n_{p}(\theta;\xi)]=0$.
\item The functions  $\it\Omega_p^n=\mathbf\Omega_p^ n(\xi)-|p|^2$ with $$e_n+\sum\limits_{p\geq1} (\it\Omega_p^ n+B_p^n) |z_p|^2+P_n\in\mathcal F_{r_n,s_n,\mathcal O_n}(\mathcal E_n,\mathcal E_n,\epsilon_n).$$
\end{enumerate}
Then there exists a real analytic symplectic transformation $$\Phi_{n+1}:D(r_{n+1},s_{n+1})\times\mathcal O_{n+1}\rightarrow D(r_n,s_n)$$
with \begin{equation}\label{measure}meas (\mathcal O_n\backslash \mathcal O_{n+1})\leq {C\gamma_{n+1}\over K_{n+1}}\end{equation}
and
\begin{eqnarray}\label{phi-1}
&&|\Phi_{n+1}-id|_{s_{n+1},D_{n+1}\times \mathcal O_{n+1}}^* \leq \epsilon_n^{\frac{1}{2}},\\
 &&\label{phi-2}\pmb{\pmb |}D(\Phi_{n+1}-id)\pmb{\pmb |}_{s_{n+1},s_{n+1},D_{n+1}\times \mathcal O_{n+1}}^{*}
\leq \epsilon_n^{\frac{1}{2}},\end{eqnarray}
 such that
$H_{n+1}=H_{n}\circ\Phi_{n+1}$ satisfies  the assumptions of $ H_n$  with $n+1$ in place of $n$.
\end{Proposition}

\section{Proof of the main results}
\subsection{Proof of Theorem \ref{kam}}
We are now in   position to prove Theorem \ref{kam}. We start with the Hamiltonian
 \begin{equation}\label{ham2}
 H_0=e_0(\theta;\xi)+ \la\omega_1,I\ra+\la\omega_2,J\ra+\sum_{p\geq 1}  (\mathbf\Omega_p^0+B_p^0(\theta;\xi))|z_p|^2+P_0(\theta,\varphi,z,\bar z;\xi)\end{equation}
which is defined on  $D(r_0,s_0)\times \mathcal O_0$, where
\[r_0=r,\, s_0=s,\gamma_0=\gamma,\, K_0=r_0^{-1}\ln\epsilon_0^{-1},e_0=0, \mathbf\Omega_p^0=\mathbf\Omega_p,B_p^0=0, P_0=P.\]
\begin{equation}\nonumber
\mathcal O_0=
\left\{\xi\in\mathcal O:
\begin{array}{ll}
|\la k,\omega_1\ra+\la l,\omega_2\ra\pm{\mathbf\Omega}^0_p|\geq  {\gamma_0\over (|k|+|l|+1)^\tau},&\\
|\la k,\omega_1\ra+\la l,\omega_2\ra\pm({\mathbf\Omega}^0_p+{\mathbf\Omega}^0_q)|\geq  {\gamma_0\over (|k|+|l|+1)^\tau},&\\
|\la k,\omega_1\ra+\la l,\omega_2\ra\pm({\mathbf\Omega}^0_p-{\mathbf\Omega}^0_q)|\geq  {\gamma_0\over (|k|+|l|+1)^\tau},&|l|+|p-q| \neq 0;
\end{array}
\right\}
\end{equation}

Then the assumption of  Proposition \ref{KAM iteration}  are satisfied for $n=0$ since $\epsilon<\epsilon_0$ and $\mathcal E_0=0$,  we thus get the symplectic transformation  $\Phi_1:D(r_1,s_1)\times\mathcal O_1\to D(r_0,s_0)$. Inductively we  obtain a sequence:
\[\Phi_{n+1}:D(r_{n+1},s_{n+1})\times\mathcal O_{n+1}\to D(r_{n},s_{n}),
\]
such that
\[\Phi^{n+1}=\Phi_1\circ\Phi_2\circ\cdots\circ\Phi_{n+1}:D(r_{n+1},s_{n+1})\times\cal
\mathcal O_{n+1}\to D(r_0,s_0),\,n\ge 0,
\]  conjugate the  Hamiltonian \eqref{ham2} to
\begin{eqnarray*}&&H_{n+1} \\ &=&e_{n+1}(\theta;\xi)+ \la\omega_1,I\ra+\la\omega_2,J\ra+\sum_{p\geq 1} (\mathbf\Omega_p^{n+1}+B_p^{n+1}(\theta;\xi))|z_p|^2 +P_{n+1}(\theta,\varphi,z,\bar z;\xi)\end{eqnarray*}
with estimates:
$$e_{n+1}+\sum\limits_{p\geq1} (\it\Omega_p^ {n+1}+B_p^{n+1}) |z_p|^2+P_{n+1}\in\mathcal F_{r_{n+1},s_{n+1},\mathcal O_{n+1}}(\mathcal E_{n+1},\mathcal E_{n+1},\epsilon_{n+1}).$$
and the  symplectic map satisfy
\begin{eqnarray*}
&&|\Phi_{n+1}-id|_{s_{n+1},D_{n+1}\times \mathcal O_{n+1}}^* \leq \epsilon_n^{\frac{1}{2}},\\
 &&\label{phi-2}\pmb{\pmb |}D(\Phi_{n+1}-id)\pmb{\pmb |}_{s_{n+1},s_{n+1},D_{n+1}\times \mathcal O_{n+1}}^{*}
\leq \epsilon_n^{\frac{1}{2}}.\end{eqnarray*}

For $n\geq 0$, by the chain rule, we get
\begin{equation}\pmb{\pmb |}D\Phi^{n+1} \pmb{\pmb |}_{s_0,s_{n+1},D_{n+1}}\leq \prod\limits_{m=1}^{n+1} \pmb{\pmb |}D\Phi_m \pmb{\pmb |}_{s_{m-1},s_m,D_m}\leq \prod\limits_{m=1}^{n+1}(1+\epsilon_{m-1}^{1\over2})\leq 2\end{equation}

\begin{eqnarray}\pmb{\pmb |}D\Phi^{n+1} \pmb{\pmb |}^{\mathcal L}_{s_0,s_{n+1},D_{n+1}}&\leq&\sum\limits_{m=1}^{n+1}(\pmb{\pmb |}D\Phi_m\pmb{\pmb |}_{s_{m-1},s_m,D_m}^{\mathcal L}\prod\limits_{j=1,j\neq m}^{n+1} \pmb{\pmb |}D\Phi_j\pmb{\pmb |}_{s_j,s_j,D_j})\nonumber \\
&\leq& \sum\limits_{m=1}^{n+1} 2\epsilon_{m-1}^{1\over2}\leq 2.\end{eqnarray}
As a consequence, we have
\begin{equation}|\Phi^{n+1}-\Phi^{n}|_{s_0,D_{n+1}}
\leq \pmb{\pmb |}D\Phi^{n} \pmb{\pmb |}_{s_0,s_n,D_n}|\Phi_{n+1}-id|_{s_{n},D_{n+1}}\leq 2\epsilon_{n}^{2\over3}\nonumber,\end{equation}
and
\begin{eqnarray*}&&|\Phi^{n+1}-\Phi^{n}|^{\mathcal L}_{s_0,D_{n+1}}\\&\leq&\pmb{\pmb |}D\Phi^{n}\pmb{\pmb |}_{s_0,s_n,D_n} |\Phi_{n+1}-id |^{\mathcal L}_{s_{n},D_{n+1}}+\pmb{\pmb |}D\Phi^{n}\pmb{\pmb |}^{\mathcal L}_{s_0,s_n,D_n}|\Phi_{n+1}-id |_{s_{n},D_{n+1}}\nonumber\\
&\leq&2|\Phi_{n+1}-id|^{*}_{s_{n},D_{n+1}}\leq 2\epsilon_n^{1\over2}.\end{eqnarray*}
%From our construction of CD-Bridge, one has
%$$\epsilon_{n}=\epsilon_{n-1}Q_{n+1}^{-2^{n+1}c\tau U}\rightarrow0,\quad(n\rightarrow\infty).$$$\llbracket $

The remaining task is to prove $C^\infty$ smoothness of  $\Phi^\infty$ on $\Theta\in \T^{2+d}$.
From the choice of parameter  $\epsilon_n$, and for any $b\in \Z^{2+d}$,
 there exists
some $N\in \N$ such that  $Q_{n+1}^{4|b|}<\epsilon_n^{-1/4}$ for any $n\geq N$, that is
$$Q_{n+1}^{4|b|}\epsilon_n^{1\over2}< \epsilon_n^{1\over4},\forall n\geq N.$$
Then according to the Cauchy estimate, one has
\begin{equation}
|{\partial^{|b|}\over\partial \Theta^b}(\Phi^{n+1}-\Phi^{n})|\leq r_{n+1}^{-|b|}|\Phi^{n+1}-\Phi^{n}|_{s_0,D_{n+1}}\leq Q_{n+1}^{4|b|}\epsilon_n^{1\over2}\leq \epsilon_n^{1\over4}
\end{equation}
and also
\begin{equation}
|{\partial^{|b|}\over\partial \Theta^b}(\Phi^{n+1}-\Phi^{n})|^{\mathcal L}\leq r_{n+1}^{-|b|}|\Phi^{n+1}-\Phi^{n}|^{\mathcal L}_{s_0,D_{n+1}}\leq Q_{n+1}^{4|b|}\epsilon_{n}^{1\over2}\leq \epsilon_n^{1\over4}.
\end{equation}
Thus $\Phi^n$ converges uniformly on $\T^{2+d}\times \{0,0,0\}\times\mathcal O_\infty$, and the limit $\Phi^\infty=\lim\limits_{n\rightarrow\infty}\Phi^{n+1}$ is  $C^\infty$ smooth on $\Theta$.
Let $\phi_H^t$ be the flow of $X_H$, since
$H\circ\Phi^{n}=H_{n}$, there is
\beq\label{5.7}
\phi_H^t\circ\Phi^{n}=\Phi^{n+1}\circ\phi_{H_{n}}^t. \eeq
 The
uniform convergence of $\Phi^{n},D\Phi^{n}$ and
$X_{H_{n}}$ implies that the limits can be taken on both sides of
(\ref{5.7}). Hence, on $D(0,0)\times{\mathcal O}$ we
get $$
\phi_H^t\circ\Phi^\infty=\Phi^\infty\circ\phi_{H_{\infty}}^t$$ and
$$
\Psi^\infty:D(0,0)\times{\mathcal O}_\infty\to
 D(r,s).
 $$

By \eqref{measure}, the total measure we excluded is
\begin{equation}
meas(\mathcal O\backslash\mathcal O^\infty)\leq\sum\limits_{n=0}^\infty {C\gamma_n\over K_n}\leq C\gamma_0.\end{equation} \qed

\subsection{Proof of Theorem \ref{thm-appli-1}:}

 As an application of Theorem \ref{kam}, we  study the equation $\eqref{equ}$  on some suitable phase space.
% Operator $\xi\triangle$ has eigenfunctions $\phi_j(t)=\sqrt{2\over {\pi}}sin( {jx})$ and eigenvalue $\xi j^2,j\geq 1$.
As it is well known, the operator $\partial_{xx}+v(x)$ has  an orthonormal
basis  $\phi_p\in L^2{[0,\pi]}$, $p\geq 1$, with corresponding eigenvalues $\Omega_p$ satisfying the asymptotics for large $p$,
\beq\label{eigenvalues}{\mathbf\Omega}_p=p^2+{1\over\pi}\int_{0}^\pi v(x)dx+o(p^{-1})\eeq

 To write $\eqref{equ}$ in
infinitely many coordinates, we make the ansatz
$$u(t,x)=\mathscr{S}z=\sum_{p\geq1} z_p(t)\phi_p(x),p\geq 1.$$
Then  $\eqref{equ}$ is written as a non-autonomous  Hamiltonian
\begin{eqnarray*}H(u)=\sum_{p\geq1} {\mathbf\Omega}_p |z_p|^2+\epsilon\int_0^\pi F(\omega t,x,\mathscr{S}z,\mathscr{S}\bar z;\xi)dx.\end{eqnarray*}
with  symplectic structure $i\sum\limits_{p\geq 1}dz_p\wedge d\bar z_p$, where $F$ is a function such that $F_{\bar u}(\Theta,x,u,\bar u;\xi)=f(\Theta,x,u,\bar u;\xi)$. Then one has a modified system
\begin{eqnarray}
\left\{\begin{array}{l}
\dot\theta=\omega_1,\\
\dot\varphi=\omega_2,\\
\dot{z}_p=-i{\mathbf\Omega}_p z_p-i\partial_{\bar{z}_p}P(\theta,\varphi,z,\bar z;\xi),\ p\geq 1,\\
\dot{\bar{z}}_p=i {\mathbf\Omega}_p\bar{z}_p+i\partial_{{z}_p}P(\theta,\varphi,z,\bar z;\xi),\ p\geq 1,
\end{array}\right.\label{hs00}
\end{eqnarray}
 We introduce auxiliary action variable $I,J$ and rewrite (\ref{hs00}) to  an autonomous  system for convenience
\begin{eqnarray}
\left\{\begin{array}{ll}
\dot\theta=\omega_1,\\
\dot\varphi=\omega_2,\\
\dot{I}=-\partial_{{\theta}}P(\theta,\varphi,z,\bar z;\xi)\\
\dot{J}=-\partial_{{\varphi}}P(\theta,\varphi,z,\bar z;\xi)\\
\dot{z}_p=-i{\mathbf\Omega}_pz_p-i\partial_{\bar{z}_p}P(\theta,\varphi,z,\bar z;\xi),\quad p\geq 1,\\
\dot{\bar{z}}_p=i {\mathbf\Omega}_p\bar{z}_p+i\partial_{{z}_p}P(\theta,\varphi,z,\bar z;\xi),\quad p\geq 1.
\end{array}\right.\label{hs01}
\end{eqnarray}
That is we consider  the Hamiltonian
\begin{eqnarray*}H&=&N+P(\theta,\varphi,z,\bar z;\xi) \\
&=&\la\omega_1,I\ra+\la\omega_2,J\ra+\sum_{p\geq1} {\mathbf\Omega}_p |z_p|^2+\epsilon\int_0^\pi F(\theta,\varphi,x,\mathscr{S}z,\mathscr{S}\bar z;\xi)dx\end{eqnarray*}
with  symplectic structure
$d I\wedge d\theta+d J\wedge d\varphi+i\sum\limits_{p\geq 1}dz_p\wedge d\bar z_p$.\\

Next let us verify that $H = N + P$ satisfies the assumptions of Theorem \ref{kam}.
Recall the  eigenvalue $\mathbf\Omega_p$  satisfy  \eqref{eigenvalues},  and $v(x)$ is independent of $\xi$,  thus one has
 $$\sup_{p\geq1}|\mathbf\Omega_p-p^2|^*_{\mathcal O}=\sup_{p\geq1}|{1\over\pi}\int_{0}^\pi v(x)dx+o(p^{-1})|^*_{\mathcal O}<{1\over 2}.$$

The regularity of the perturbation is given by the following basic Lemma.
\begin{Lemma} \label{per NLS}
Suppose that $v(x)$ is real analytic in $x$, then for small enough $r,s,a,\rho> 0$, $X_P$ is real analytic as a map
from some neighborhood of the origin in $\ell^{a,\rho}$ to $\ell^{a,\rho}$, in particularly
$$|X_P|_{s,D(r,s)\times \mathcal O}^*\leq\epsilon.$$
\end{Lemma}
\begin{proof}From the hypotheses that $v(x)$ is real analytic, it follows that the eigenfunctions $\phi_p, p\geq1$
are analytic. Let $u(t,x)=\sum_{p\geq1} z_p\phi_i$ and $\bar u(t,x)=\sum_{p\geq1} \bar z_p\phi_i$ with $z,\bar z\in \ell^{a,\rho}$.
Since  $$\partial_{{z}_p}P(\theta,\varphi,z,\bar z;\xi)=-i\epsilon\int_{0}^\pi f(x,\sum_{p\geq1} z_p\phi_i ,\sum_{p\geq1} \bar z_p\phi_i;\xi)\phi_p dx.$$  It follows that
 $|X_P|^*_{s,D(r,s)\times\mathcal O}\leq \epsilon.$\end{proof}\qed

Thus  Theorem \ref{kam} is applicable, and  the  system $\eqref{hs00}$ is conjugate to
\begin{eqnarray}
\left\{\begin{array}{l}
\dot\theta=\omega_1,\\
\dot\varphi=\omega_2,\\
\dot{z}_p=-i( {\mathbf\Omega}_p^*+B^*_p(\theta;\xi))z_p-i\partial_{\bar{z}_p}P^*(\theta,\varphi,z,\bar z;\xi),\ p\geq 1,\\
\dot{\bar{z}}_p=i( {\mathbf\Omega}^*_p+B_p^*(\theta;\xi))\bar{z}_p+i\partial_{{z}_p}P^*(\theta,\varphi,z,\bar z;\xi),\ p\geq 1,
\end{array}\right.\label{hs01}
\end{eqnarray}
by  $\Phi: \T^{2+d}\times \{0,0\}\rightarrow D(r,s)$.
Since $\Theta=(\theta,\varphi),\omega=(\omega_1,\omega_2)$, $(\Theta^*(0)+\omega t; 0; 0)$ is a solution of \eqref{hs01}. Let $(\Theta(t); z(t); \bar z(t))=\Phi(\Theta^*(0)+
\omega t; 0; 0)$, by Theorem \ref{kam},
$z(t) = g(\Theta^*(0) + \omega t)$  is $C^\infty$ smooth in $t$. Then
$(\Theta(t); z(t); \bar z(t)) = (\Theta^*(0) +\omega t; g(\Theta^*(0) + \omega t); \bar g(\Theta^*(0) + \omega t))$
is a solution of $\eqref{hs00}$ for any
$\xi\in\mathcal O_\gamma$, and  the equation $\eqref{equ}$ has a quasi periodic solution $$u(t,x)=\mathscr{S}z=\sum_{j\geq 1}g_j(\Theta^*(0) + \omega t)\phi_j(x)=\sum_{j\geq 1}g_j(\Theta^*(0) + \xi\bar\omega t)\phi_j(x)$$ which is  $C^\infty$ smooth in $t$. Thus we have our result of Theorem \ref{thm-appli-1}.
%
%\begin{Remark}The transformation $\Phi_n$ is composed by a symplectic conjugation $\Psi_n$, a series finite  induction $\tilde\phi_i^n$  and $\Psi_n^{-1}$.
%\end{Remark}

\section{Proof of Proposition \ref{KAM iteration}}
This main proposition is proved by KAM iteration. As we mentioned before, finite many iterations are required. Since our homological
equation depends on the angle $\theta$, it will be  hard for us to solve this equation. Thus in the following, we first introduce an abstract
result on the homological equation,   a finite  iteration lemma will be given and then we complete the proof of Proposition \ref{KAM iteration}.

\subsection{Homological equation}\label{homoeq}
During  the KAM iteration, a more complicated homological  equation come out, namely:
 \begin{equation}\label{3.1}
-i\la \partial_\theta F_{\ell}(\theta,\varphi),\omega_1\ra-i\la \partial_\varphi F_{\ell}(\theta,\varphi),\omega_2\ra +\la \ell,\mathbf\Omega+B(\theta)\ra F_{\ell}(\theta,\varphi)=R_{\ell}(\theta,\varphi),
\end{equation}
where  $\ell\in \Z^{\N}$ with $|\ell|=1$ or $2$. As $B(\theta)$ is of size $\varepsilon_0$ and $(\omega_1,\omega_2)$ is Liouvillean, \eqref{3.1} will have no analytic solution.
Actually,  Wang-You-Zhou \cite{WYZ} met similar problem when they consider response solutions of harmonic oscillators, then the first and second Melnikov conditions are required for $\forall \ell\in\Z^d$ with $ |\ell|=1$ or $2$,
\begin{equation}\label{2.5}|\la k, \bar\omega_1\ra+\la\ell, \mathbf\Omega(\lambda)\ra|\geq {\gamma\over (|k|+|\ell|)^\tau}.\end{equation}
 For small divisor as above, the key  observation is the following : for a very large and specialized truncation $K$, $|\la k,\bar\omega_1\ra+\la\ell, \mathbf\Omega(\lambda)\ra|$
  has an  uniform relative large lower bound for any $k$ such that
   $|k|\leq K$ (Lemma 3.2 of \cite{WYZ}). %, which is right only for finite dimension system.
  With this  observation, they  construct $C^\omega$ smooth  response solution for any $\alpha\in \mathbb R\backslash\mathbb Q$.
  However, this phenomena is not right for the problem we meet, since in our case, $\ell\in \Z^N$ with $|\ell|=1$ or $2$, therefore  there are infinitely many choices of $l$.
  Similar problem was also met during the work of Krikorian-Wang-You-Zhou \cite{KWZY}, where they solved  the following homological equation
$$\partial_{\bar\omega_1} h(\theta,\phi)+(\rho+b(\phi))\frac{\partial h}{\partial \theta}=f(\theta,\phi)$$
with Melnikov condition
 $$|\la k,\bar\omega_1\ra+\rho l|\geq {\gamma\over (|k|+|l|)^\tau},\quad\forall \ell\in \Z\backslash\{0\},k\in\Z^2.$$
by the method of diagonally dominant (Proposition 4.1  of \cite{KWZY}). In this paper, we will borrow some method developed in \cite{KWZY}, but gave more concise argument and uniform ways to deal with this kind of equations, we believe it will have more  applications.  Also we stress that we can deal with multi-frequencies, while all the former results were restricted to one frequency (thus two frequency in the continuous case).

 \begin{Proposition}\label{basic lemma}
 Let $\gamma>0,\lambda\geq1,\tau>d+6$ and $\zeta\in \R\backslash\{0\}$,  $0<\tilde\sigma<\tilde{\mathbbm r}<{\mathbbm r}$, $0<\eta_1,\eta_2,\tilde \eta<1$.
  Consider the equation
 \begin{equation}\label{basic homological equation}-i\la \partial_\theta F(\theta,\varphi),\bar\omega_1\ra-i\la \partial_\varphi F(\theta,\varphi),\bar\omega_2\ra+(\zeta+B(\theta)+b(\theta)) F(\theta,\varphi)=R(\theta,\varphi)\end{equation}
 with $[B(\theta)]=[b(\theta)]=0$. Suppose that  $|B(\theta)|_{\mathbbm r}\leq \eta_1, |b(\theta)|_{\tilde{\mathbbm  r}}\leq  \eta_2, |R(\theta)|_{\tilde {\mathbbm r}}\leq \tilde\eta$ which furthermore satisfy the following condition:
 \begin{enumerate}
\item $\eta_1 e^{-|{\mathbbm  r}-{\tilde{\mathbbm  r}}|Q_{n+1}}\leq \eta_2$\label{cond1}

\item  $2\eta_1 Q_{n+1}^3{\tilde {\mathbbm  r}}   \leq ({\mathbbm  r}-{\tilde {\mathbbm  r}})^4$,\label{cond2}

\item $K={1\over{\tilde \sigma}}\ln{1\over \tilde\eta}\leq({\gamma\over 2\eta_2})^{1\over \tau+2}$\label{cond3}

 \item $|\la k,\bar\omega_1\ra+\la l,\bar\omega_2\ra+\zeta|\geq{\gamma\lambda\over (|k|+|l|+1)^{\tau}},\quad  |k|+|l|\leq K$.\label{cond4}
\end{enumerate}
Then the equation \eqref{basic homological equation} has an approximate solution $F(\theta,\varphi)$  with estimation
\begin{equation}
|F(\theta,\varphi)|_{{\tilde{\mathbbm  r}}-{\tilde \sigma}}\leq {c\tilde\eta\over \lambda\gamma {\tilde \sigma}^{3+\tau}}.\end{equation}
Moreover, the error term satisfies
\begin{eqnarray}\label{errorper}
|\breve R|_{{\tilde{\mathbbm  r}}-{\tilde \sigma}}&=&|e^{i\mathcal B(\theta)}(I-\mathcal T_K)(e^{-i\mathcal B(\theta)}R(\theta,\varphi)-(b(\theta)+(I-\mathcal T_K) B(\theta))F(\theta,\varphi))|_{{\tilde {\mathbbm  r}}-{\tilde \sigma}}\nonumber\\
&\leq& {\tilde \eta^2\over\lambda\gamma{\tilde \sigma}^{3+\tau}},
\end{eqnarray}
where $\mathcal B(\theta)$ is the solution of  $\la\partial_\theta \mathcal B(\theta),{\bar\omega_1}\ra=\mathcal T_{ Q_{n+1}}B(\theta).$
\end{Proposition}

%
%\begin{Remark}
%In KAM iteration,  $\eta_1$ is of size $\epsilon_0$, $\eta_2$ is of size $\epsilon_n$ and $\tilde\eta$ is smaller than $\epsilon_n$.
%\end{Remark}
%
%
%
\begin{Remark}
The above Proposition \ref{basic lemma} holds irrespectively of any arithmetical property of $\bar \omega=({\bar\omega_1},\bar\omega_2)$.
\end{Remark}

\begin{Remark}
The assumption $2$ means there is a quick shrink from analytic radius ${\mathbbm  r}$ to ${\tilde {\mathbbm  r}}$. We have this assumption if  ${\mathbbm  r}={r_0\over Q_n^4}$ and ${\tilde {\mathbbm  r}}\leq{r_0\over Q_{n+1}^4}$.
\end{Remark}

\begin{proof}
Let $\tilde F(\theta,\varphi)=F(\theta,\varphi)e^{i{\mathcal B(\theta)}}$, where $\mathcal B(\theta)$ is the solution of  $$\la\partial_\theta \mathcal B(\theta),{\bar\omega_1}\ra=\mathcal T_{ Q_{n+1}}B(\theta).$$
Then we consider equation
\begin{equation}\label{basic homological equation II}\mathcal T_K(-i\la \partial_\theta \tilde F,\bar\omega_1\ra-i\la \partial_\varphi \tilde F,\bar\omega_2\ra+(\zeta+ \tilde b(\theta)) \tilde F(\theta,\varphi))=\mathcal T_K{\tilde R}(\theta,\varphi)\end{equation}
with $\tilde b(\theta)=(I-\mathcal T_K) B(\theta)+ b(\theta),\,\tilde R(\theta,\varphi)=e^{i\mathcal B(\theta)}R(\theta,\varphi)$.
By the assumption \ref{cond1}, one has \begin{eqnarray}
|\tilde b(\theta)|_{\tilde {\mathbbm  r}}&\leq& e^{-|{\mathbbm  r}-{\tilde {\mathbbm  r}}|Q_{n+1}}\eta_1+\eta_2\leq 2\eta_2.
\end{eqnarray}
 Since we only seek approximation of \eqref{basic homological equation},  we set $\mathcal T_K {\tilde F}=\tilde F$ for convenience below.

%Notice $non-superlouvillean property of \alpha $,
 In order to control the norm of $\tilde R$, which is a  conjugation of  $R$ by $e^{i\mathcal B}$, it is  sufficient to estimate $\texttt{Im }\mathcal B$.
As argued in Lemma 4.1 of  \cite{YZ2},  let $\theta=x+iy$ and recall  $\mathcal B(\theta)=:\sum\limits_{0<|k|<Q_{n+1}}{\hat B(k)\over i\la k,\bar\omega_1\ra}e^{i\la k,\theta\ra}$ with $|\hat B(k)|\leq |B|_{\mathbbm  r}e^{-|k|{\mathbbm  r}}$, we define
\begin{equation}\mathcal B^1=:\sum\limits_{0<|k|<Q_{n+1}}{\hat B(k)\over i\la k,\bar\omega_1\ra}e^{i\la k,x\ra}, \,\mathcal B^2=\mathcal B-\mathcal B^1.\end{equation}
Since $B(\theta)$ is real analytic, one has $\texttt{Im }\mathcal B^1=0$ and $\texttt{Im } \mathcal B=\texttt{Im }\mathcal B^2$,
 \begin{eqnarray}
|\texttt{Im} \mathcal B|_{\tilde {\mathbbm  r}}&=&|\texttt{Im} \mathcal B^2|_{\tilde {\mathbbm  r}}\leq \sum\limits_{1\leq |k|< Q_{n+1}}|{\hat B^n(k)\over \la k,\bar\omega_1\ra}||e^{-\la k,y\ra}-1|\nonumber\\
&\leq& |B(\theta)|_{\mathbbm  r} Q_{n+1}\sum\limits_{1\leq |k|< Q_{n+1}}2e^{-|k|({\mathbbm  r}-{\tilde {\mathbbm  r}})}|k|{\tilde{\mathbbm  r}}\nonumber\\
 &\leq& {\eta_1Q_{n+1}^3{\tilde {\mathbbm  r}}\over  ({\mathbbm  r}-{\tilde {\mathbbm  r}})^4}
   \leq{1\over2},
\end{eqnarray}
where the last inequality is given by assumption \ref{cond2}. As a consequence, we have
\begin{equation}\label{per-homo}|\tilde R(\theta,\varphi)|_{\tilde {\mathbbm  r}}\leq e^{|\texttt{Im} \mathcal B|_{\tilde {\mathbbm  r}}}\cdot |R(\theta,\varphi)|_{{\tilde {\mathbbm  r}}}\leq 2\tilde\eta.\end{equation}

Now we start to solve the equation \eqref{basic homological equation II}.

Let $$\tilde F(\theta,\varphi)=\sum_{l}\tilde F_l(\theta)e^{i\la l,\varphi\ra},\tilde F_l(\theta)=\sum_{|k|\leq K-|l|}\tilde F_{l}^ke^{i\la k,\theta\ra},$$
$$\tilde R(\theta,\varphi)=\sum_{l}\tilde R_l(\theta)e^{i\la l,\varphi\ra},\tilde R_l(\theta)=\sum_{|k|<K-|l|}\tilde R_{l}^ke^{i\la k,\theta\ra}.$$
Then \eqref{basic homological equation II} is equivalent to the equations below for  $|l|\leq K$,

 \begin{equation}\label{basic homological equation III}(A_l+D_l)\tilde{\mathcal F}_l=\tilde{\mathcal R}_l,\end{equation}
 where
$$A_{l}=diag(\cdots,\la k,\bar\omega_1\ra+\la l,\bar\omega_2\ra+\zeta,\cdots)_{|k|<K-|l|}, D_l=(\hat{\tilde b}(k_i-k_j))_{|k_i|,|k_j|\leq {K-|l|}} $$

\[\tilde{\mathcal F}_l=({\cdots, {\tilde  F}^k_l},\cdots)_{|k|\leq K-|l|}^T,
\tilde{\mathcal R}_l=({\cdots, {\tilde  R}^k_l},\cdots)_{|k|\leq K-|l|}^T.\]

Let $M_{l,r'}=diag(\cdots,e^{|k|r'},\cdots)_{|k|\leq K-|l|}$ for any $r'\leq {\tilde {\mathbbm  r}}$, then \eqref{basic homological equation III} is equivalent to
 $$M_{l,r'}(A_l+D_l)M_{l,r'}^{-1}M_{l,r'}\tilde{\mathcal F}_l=M_{l,r'}\tilde{\mathcal R}_l.$$
We rewrite is to be
\begin{equation}(A_l+D_{l,r'})\tilde{\mathcal F}_{l,r'}=\tilde{\mathcal R}_{l,r'}\end{equation}
 where $D_{l,r'}=M_{l,r'}D_l M_{l,r'}^{-1}, \tilde{\mathcal F}_{l,r'}=M_{l,r'}\tilde{\mathcal F}_l,\tilde{\mathcal R}_{l,r'}=M_{l,r'}\tilde{\mathcal R}_l.$
A simple calculous show that $\|D_{l,r'}\|\leq (K-|l|)^2\eta_2$.

%By assumption 3 and 4,
%With a simple calculous, $$|\la k,\bar\omega_1\ra+\la l,\bar\omega_2\ra+\zeta|\geq {\gamma\lambda\over (|k|+|l|+1)^{\tau}}\geq {\gamma\lambda\over K^{\tau}}\geq(K-|l|)\eta_2^{1\over2}. $$

%Thus the operator $(A_l+D_{l,r'})$ is diagonal dominant and one has

By assumption \ref{cond3} and  \ref{cond4}, we have
 $$|\la k,\bar\omega_1\ra+\la l,\bar\omega_2\ra+\zeta|\geq {\gamma\lambda\over (|k|+|l|+1)^{\tau}}\geq {\gamma\lambda\over K^{\tau}}\geq(K-|l|)\eta_2^{1\over2}. $$
for all $|k|+|l|< K$. As a result,  the diagonally
dominant  operators $A_l+D_{l,r'}$ has
a bounded inverse and $\|(I+A_l^{-1}D_{l,r'}\|_{op( l^{1})}<2$ for $r'={\tilde{\mathbbm r}}-{\tilde \sigma}$,  where $\|\cdot\|_{op(l^1)}$ denotes the operator norm associated to the $l^1$ norm $|u|_{l^1}=\sum_{|k|<K-| l |} | u^k|$. To see this, one can compute

\begin{eqnarray}\|A_{l}^{-1}D_{l,r'}\|_{op(\ell^1)}
&\leq& \max_{|k_i|\leq K} \sum_{|k_j|\leq K}{ |\hat {\tilde b}(k_i-k_j)|e^{(|k_i|-|k_j|)r'}\over |\la k_i,\bar\omega_1\ra+\la l,\bar\omega_2\ra+\zeta|}\\
&\leq&  \max_{|k_i|\leq K} \sum_{|k_j|\leq K} { e^{-|k_i-k_j|{\tilde {\mathbbm  r}}+(|k_i|-|k_j|)r'} K^{\tau}\eta_2 \over \gamma \lambda}\nonumber\\
 &\leq &{ K^{\tau+2}\eta_2 \over \gamma \lambda} \leq{ K^{\tau+2}\eta_2 \over \gamma} \leq  1/2,\nonumber\end{eqnarray}
 where the last inequality follows from the assumption \ref{cond3}.

Since $$(I_l+A_l^{-1}D_{l,r'})^{-1}=\sum\limits_{i=0}^\infty (-1)^n (A_l^{-1}D_{l,r'})^n.$$
one has  $\|(I_l+A_l^{-1}D_{l,r'})^{-1}\|_{op(\ell^1)}\leq 2.$
As a conclusion,  the approximate solution
$\mathcal F_{l,r'}=(I+A_l^{-1}D_{l,r'})^{-1}A_l^{-1}\mathcal F_{l,r'}$ is regular with
\begin{eqnarray*}
&&| \tilde F(\theta,\varphi)|_{{\tilde {\mathbbm  r}}-{\tilde \sigma}}\\
&\leq&\sum_{|k|+|l|<K}|{\tilde F}_l^k|e^{(|k|+|l|)({\tilde {\mathbbm  r}}-{\tilde \sigma})}\\
&=&\sum_{|l|\leq K}(\sum_{|k|<K-|l|}| {\tilde F}^k_l|e^{|k|({\tilde {\mathbbm  r}}-{\tilde \sigma})})e^{|l|({\tilde {\mathbbm  r}}-{\tilde \sigma})}=\sum_{|l|\leq K}|\tilde{\mathcal F}_{l,{\tilde {\mathbbm r}}-{\tilde \sigma}}|_{l^1}e^{|l|({\tilde {\mathbbm  r}}-{\tilde \sigma})}\\
&\leq&\sum_{|l|\leq K}\|(I_l+A_l^{-1}D_{l,r'})^{-1}\|_{op(l^1)}|A_l^{-1}\tilde{\mathcal R}_{l,\tilde {\mathbbm  r}-{{\tilde \sigma}}}|_{l^1}e^{|l|({\tilde {\mathbbm  r}}-{\tilde \sigma})}\\
&\leq&2\sum_{|l|\leq K}\sum_{|k|<K-|l|}{(|k|+|l|+1)^\tau\over\lambda \gamma} |{\tilde R}^k_l|e^{(|k|+|l|)({{\tilde {\mathbbm  r}}-{{\tilde \sigma}}})}\\
&\leq&{c\tilde\eta\over\lambda \gamma {\tilde \sigma}^{3+\tau}}.\end{eqnarray*}
Let $F=e^{-i\mathcal B(\theta)}\tilde F$, like \eqref{per-homo}, one has
$$| F(\theta,\varphi)|_{{\tilde {\mathbbm  r}}-{\tilde \sigma}}\leq e^{|\texttt{Im} \mathcal B(\theta)|_{{\tilde {\mathbbm  r}}}}\cdot |\tilde F(\theta,\varphi)|_{{\tilde {\mathbbm  r}}-{\tilde\sigma}}\leq {c\tilde\eta\over\lambda \gamma {\tilde \sigma}^{3+\tau}}.$$
Moreover, one can compute  $F$ will solve equation \eqref{basic homological equation} with error term $\breve R$,
$$\breve R=e^{i\mathcal B(\theta)}(I-\mathcal T_K)(e^{-i\mathcal B(\theta)}R(\theta,\varphi)-(b(\theta)+(I-\mathcal T_K) B(\theta))F(\theta,\varphi)).$$
The estimation on $\breve R$ is a direct computation.\qed
\end{proof}

\subsection{A finite KAM  induction}
We will prove Proposition \ref{KAM iteration} by induction. We start with the  Hamiltonian \begin{eqnarray*}\tilde H_0={\tilde  e_0}(\theta;\xi)+\la\omega_1,I\ra+\la\omega_2,J\ra+\sum\limits_{p\geq1}(\tilde{ \mathbf \Omega}_p^0(\xi)+ B_p(\theta;\xi))|z_p|^2+\tilde P_0(\theta,\varphi,z,\bar z;\xi)\end{eqnarray*}
 on  $D(\tilde r_0, \tilde s_0)\times\tilde{ \mathcal O}_0,\tilde\epsilon_0=\epsilon_n,\tilde{\mathcal E}_0=\mathcal E_n$, where \begin{equation}\label{newbegining}\tilde r_0=2r_{n+1},\tilde s_0=s_n,\tilde{\mathcal O}_0=\mathcal O_n,\end{equation}
% $B_p(\theta),\tilde P_0$ are real analytic on $\theta\in \T^2$,
  and $[B_p(\theta;\xi)]=0$,
  $$\tilde e_0+\sum\limits_{p\geq1}(\tilde{\it\Omega}_p^
 {0}+ B_p) |z_p|^2+ \tilde P_0\in\mathcal F_{\tilde r_0,\tilde s_0,\tilde{\mathcal O}_0}(\mathcal E_n,\mathcal E_n,{\tilde\epsilon}_0).$$

  %$$ |\tilde e_0|_{r_n,\tilde{\mathcal O}_0}\leq \mathcal E_n,|B_p|_{r_n,\tilde{\mathcal O}_0}\leq \mathcal E_n,|\tilde{\it\Omega}^0_p(\xi)|^*_{\tilde{\mathcal O}_0}\leq  \mathcal E_n+{1\over4}, |X_{\tilde P_0}|_{\tilde r_0,\tilde s_0}\leq {\tilde\epsilon}_0=\epsilon_n.$$

Let $N=2+[{2^{n+1}c{\tau}U\ln Q_{n+1}\over 6\tau+9}]$ and $K^n=({\gamma_0\over 4\epsilon_n})^{1\over \tau+2}$, where the former will control the steps of finite KAM iteration and the later is the upper bound control in assumption $4$ of Proposition \ref{basic lemma}. Define  the following iteration sequence  for $j=1,2,3,\cdots, N$,
\begin{equation}\label{minor iteration sequence}
\begin{array}{ll}
{\tilde \epsilon}_j=({\tilde \epsilon}_{j-1})^{\frac{4}{3}}=({\tilde\epsilon}_0)^{(\frac{4}{3})^j},&\tilde {\mathcal E}_j=\sum\limits_{q=0}^{j-1}\tilde\epsilon_q,\\
\tilde\gamma_j= \gamma_n-2\sum\limits_{q=0}^{j-1}{\tilde\epsilon_q^{3\over4}},&\tilde\eta_j=({\tilde\epsilon}_{j-1})^{\frac{1}{3}},\\
\tilde K_j=\tilde\sigma_j^{-1}\ln\tilde\epsilon_{j-1}^{-1},&\tilde\sigma_j={1\over5}{\tilde r_0 \over 2^{j+1}},\\
{\tilde s}_j=\tilde\eta_{j} \tilde s_{j-1},& \tilde r_{j}=\tilde r_{j-1}-5\tilde\sigma_j,\\
\tilde D_j=D(\tilde r_j,\tilde s_j).&
\end{array}
\end{equation}
We also set $\tilde K_0=K_n$.  Then we have the following iteration lemma:

 \begin{Lemma}\label{Lem4.2}
The following holds for $0\leq j\leq N-1$. Suppose  the  Hamiltonian
\begin{equation}\tilde H_{j}= \tilde e_j(\theta;\xi)+\la\omega_1,I\ra+\la\omega_2,J\ra+\sum\limits_{p\geq1} (\tilde{\mathbf\Omega}_p^j(\xi)+ B_p(\theta;\xi)+b_p^j(\theta;\xi)) |z_p|^2+\tilde P_j(\theta,\varphi,z,\bar z;\xi)\end{equation}
is defined on $\tilde D_{j}\times \tilde{ \mathcal  O}_{j}$ with $B_p\in \mathscr B_{r_n}(\mathcal O_n)$ and $[B_p(\theta;\xi)]=0$, which furthermore satisfy
\begin{enumerate}
 \item For any $\xi\in\tilde {\mathcal O}_j$, $|k|+|l|\leq \tilde K_j$ and $p, q\geq 1$ there is
\begin{eqnarray}\label{melnikov}
&&|\la k,\omega_1\ra+\la l,\omega_2\ra\pm\tilde{\mathbf\Omega}_p^j|\geq  {\tilde \gamma_j\over (|k|+|l|+1)^\tau},\nonumber\\
&&|\la k,\omega_1\ra+\la l,\omega_2\ra\pm(\tilde{\mathbf\Omega}_p^j+\tilde{\mathbf\Omega}_q^j)|\geq  {\tilde \gamma_j\over (|k|+|l|+1)^\tau}, \\
&&|\la k,\omega_1\ra+\la l,\omega_2\ra\pm(\tilde{\mathbf\Omega}_p^j-\tilde{\mathbf\Omega}_q^j)|\geq  {\tilde \gamma_j\over (|k|+|l|+1)^\tau},|l|+|p-q|\neq0.\nonumber
\end{eqnarray}
 \item The functions $b_p^j\in\mathscr B_{\tilde r_j}(\tilde{\mathcal O}_j)$  have average zero: $[b^j_{p}(\theta;\xi)]=0$,  and there is $$\tilde e_j+\sum\limits_{p\geq1}(\tilde{\it\Omega}_p^
 {j}+ b_p^{j}) |z_p|^2+ \tilde P_{j}\in\mathcal F_{\tilde r_{j},\tilde s_{j},\tilde{\mathcal O}_{j}}(\mathcal E_n+\tilde{\mathcal E}_{j},\tilde{\mathcal E}_{j},\tilde\epsilon_{j})$$
  \end{enumerate}
\noindent Then there exists a subset $\tilde{ \mathcal  O}_{j+1}=\tilde{ \mathcal  O}_{j}\backslash\tilde{\mathcal R}^{j+1}$ with $meas(\tilde{\mathcal R}^{j+1})\leq{C\tilde\gamma_{j+1}\over \tilde K_{j+1}^2}$,
and a symplectic transformation $\phi_{F_{j}}^t:\tilde D_{j+1}\times \tilde{ \mathcal  O}_j\rightarrow \tilde D_j$ with estimate
 \begin{eqnarray}\label{hmsol}
|\phi_{F_j}^t-id|^*_{{\tilde s_{j+1}},\tilde D_{j+1}\times\tilde{ \mathcal  O}_{j}} &\leq& \tilde\epsilon_{j}^{3\over4},\\
\pmb{\pmb |} D\phi_{F_j}^t-Id \pmb{\pmb |}^*_{\tilde s_{j+1},{\tilde s_{j+1}},\tilde D_{j+1}\times \tilde{ \mathcal  O}_{j}}&\leq& \tilde\epsilon_{j}^{3\over4},
\end{eqnarray}
such that  $\tilde H_{j+1}=\tilde H_j\circ\phi_{F_j}^1$  satisfies  the assumptions of $\tilde H_j$ with $j+1$ in place of $j$.
 \end{Lemma}
\begin{Remark}
The crucial point for us is that the functions $B_p(\theta;\xi),p\geq 1$ are fixed in the  iteration.
\end{Remark}
%\begin{Remark}
%By \eqref{key inequality}, which is required by Proposition \ref{basic lemma}, this Lemma cannot work infinite many steps.
%\end{Remark}

Once we have Proposition \ref{basic lemma},  the proof of this Lemma is standard KAM,  we leave it to  the appendix.

\subsection{The Construction of $\Phi_{n+1}$ and  $H_{n+1}$}

Now we are going to finish the proof of Proposition \ref{KAM iteration}.
  As a beginning, we fix $\tilde H_0=H_n$ with 
  \begin{eqnarray*}\tilde H_0={\tilde  e_0}+\la\omega_1,I\ra+\la\omega_2,J\ra+\sum_{p\geq1}(\tilde{ \mathbf \Omega}_p^0+ B_p+ b_p^0)|z_p|^2+\tilde P_0\end{eqnarray*}
 on  $D(\tilde r_0, \tilde s_0)\times\tilde{ \mathcal O}_0$, where
 $$
\tilde\epsilon_0=\epsilon_n,\tilde r_0=2r_{n+1},\tilde s_0=s_n,
\tilde K_0=K_n,\tilde{\mathcal O}_0=\mathcal O_n,\tilde e_0=e_n(\theta;\xi),
$$
$$
B_p=B_p^{n}(\theta;\xi),  b_p^0=0,
\tilde{\mathbf\Omega}_p^0=\mathbf\Omega_p^n(\xi), \tilde P_0=P_n(\theta,\varphi,z,\bar z;\xi).
$$

 By the assumption of Proposition \ref{KAM iteration},  $\epsilon_n=\epsilon_{n-1}\cdot Q_{n+1}^{-2^{n+1}c\tau U}$, $r_{n+1}={r_0\over4Q^{4}_{n+1}}$, the truncation parameter  $K_n=r_0^{-1}40Q_{n+1}^{4}\ln\epsilon_n^{-1},$ and the Diophantine condition for $\xi\in{\mathcal O}_n$  is given,
 thus the assumptions of iteration Lemma \ref{Lem4.2} are satisfied with $j=0$.
 Inductively, we iterate Lemma \ref{Lem4.2} $N$ times, we arrive at parameter set $\tilde{\mathcal O}_N$ and the Hamiltonian
\begin{equation}\tilde H_{N}=\tilde e_N+ \la\omega_1,I\ra+\la\omega_2,J\ra+\sum\limits_{p\geq 1}(\tilde{\mathbf\Omega}_p^N+B_p+ b_p^N) |z_p|^2+\tilde P_N\end{equation}
 on  $\tilde D_N\times\tilde{ \mathcal  O}_{N-1}$ with
 \begin{equation}\label{estimate with N iteration}
 \tilde e_N+\sum\limits_{p\geq1}(\tilde{\it\Omega}_p^N+  b_p^N) |z_p|^2+
 {\tilde P}_N\in{\mathcal F}_{\tilde r_N,\tilde s_N,\tilde{\mathcal O}_{N-1}}(\mathcal E_n+\tilde{\mathcal E}_N,\tilde{\mathcal E}_N,\tilde\epsilon_N)\end{equation}

First we have  the following observation: \begin{Lemma}\label{ite}
We can select
\begin{equation}\label{resonant}\mathcal O_{n+1}=\tilde{\mathcal O}_N.\end{equation}
\end{Lemma}
\begin{proof}
 By \eqref{minor iteration sequence}, we have
$$\gamma_{n+1}=\gamma_0-3\sum\limits_{i=0}^{n}{\epsilon_i^{1\over2}}=\gamma_n-3{\epsilon_n^{1\over2}}\leq \gamma_n-2\sum\limits_{i=0}^{N}{\tilde\epsilon_i^{1\over2}}=\tilde \gamma_N,$$  and $$K_{n+1}= {40\ln \epsilon^{-1}_{n+1}\over r_{n+2}}\leq {{\ln {\tilde\epsilon_0}^{({4\over3})^{N-1}}\over r_{n+1} ({1\over2})^{N-1}}}=\tilde K_N.$$
 Thus for any $|k|+|l|\leq K_{n+1}$ we have
\begin{eqnarray*}\label{newstart of next large iteration}
&|\la k,\omega_1\ra+\la l,\omega_2\ra\pm\tilde {\mathbf\Omega}_p^{N}|\geq  {\tilde \gamma_N\over (|k|+|l|+1)^\tau}\geq  {\gamma_{n+1}\over (|k|+|l|+1)^\tau},&\nonumber\\
&|\la k,\omega_1\ra+\la l,\omega_2\ra\pm (\tilde{\mathbf\Omega}_p^{N}+\tilde {\mathbf\Omega}_q^{N})|\geq  {\tilde \gamma_N\over (|k|+|l|+1)^\tau}\geq  {\gamma_{n+1}\over (|k|+|l|+1)^\tau},&\\
&|\la k,\omega_1\ra+\la l,\omega_2\ra\pm (\tilde{\mathbf\Omega}_p^{N}-\tilde {\mathbf\Omega}_q^{N})|\geq  {\tilde \gamma_N\over (|k|+|l|+1)^\tau}\geq  {\gamma_{n+1}\over (|k|+|l|+1)^\tau},&|l|+|p- q|\neq0,\nonumber
\end{eqnarray*}
 which just ensures us to choose \begin{equation*}\mathcal O_{n+1}=\tilde{\mathcal O}_N.\end{equation*} \qed
\end{proof}

Recall \eqref{minor iteration sequence} and \eqref{infinite iteration sequence}, one has
\[
\begin{array}{ll}
&\tilde r_N=\tilde r_0-\sum_{j=0}^{N-1}{\tilde r_0 \over 2^{j+2}}\geq {\tilde r_0\over2}=r_n\geq r_{n+1},\\
&\tilde s_N=\sum_{j=0}^{N-1}(({\tilde\epsilon}_0)^{{1\over3}(\frac{4}{3})^j}) \tilde s_0= \epsilon_{n}^{({4\over3})^{2+[{2^{n+1}c{\tau}U\ln Q_{n+1}\over 6\tau+9}]}-1}\cdot s_n=s_{n+1},
\end{array}
\]
and then $\tilde H_N$ is regular on $D(r_{n+1},s_{n+1})$. 
Thus we can fix $$\Phi_{n+1}=\phi_{F_0}\circ\phi_{F_1}\circ\ldots\circ\phi_{F_{N-1}}:\tilde D_N\times\tilde{\mathcal O}_{N-1}\to\tilde D_0,$$ which now is defined  on  $D(r_{n+1},s_{n+1})\times\mathcal O_{n+1}\rightarrow D(r_n,s_n)$.
We also fix $H_{n+1}=\tilde H_N$ with
\begin{equation}
e_{n+1}=\tilde e_N, B_p^{n+1}=B_p^{n}+b^N_p,
\mathbf\Omega_p^ {n+1}=\tilde{\mathbf\Omega}_p^{N}, P_{n+1}=\tilde P_{N}.
  \end{equation}
  We have our Proposition once we have following estimate.
\begin{Lemma}\label{ite-2}
We have the following estimate:
$$ e_{n+1}+\sum\limits_{p\geq1} (\it\Omega_p^ {n+1}+B_p^{n+1}) |z_p|^2+P_{n+1}\in\mathcal F_{r_{n+1},s_{n+1},\mathcal O_{n+1}}(\mathcal E_{n+1},\mathcal E_{n+1},\epsilon_{n+1}).$$
 and furthermore we have
 \begin{eqnarray}
\label{phi-1-1} |\Phi_{n+1}-id|_{s_{n+1},D_{n+1}\times \mathcal O_{n+1}}^*& \leq& \epsilon_n^{\frac{1}{2}},\\
 \label{phi-1-2}\pmb{\pmb |}D(\Phi_{n+1}-id)\pmb{\pmb |}_{s_{n+1},s_{n+1},D_{n+1}\times \mathcal O_{n+1}}^{*}
&\leq& \epsilon_n^{\frac{1}{2}},\end{eqnarray}
\end{Lemma}
\begin{proof}
First we note \begin{eqnarray*}({4\over3})^N-1 &\geq& ({4\over3})^{N-1}\geq ({4\over3})^{2^{n+1}c{\tau}U\ln Q_{n+1}}=Q_{n+1}^{2^{n+1}c\tau U\ln{ 4\over3}},
\end{eqnarray*}
and   $\tilde\epsilon_0=\epsilon_n\leq \epsilon_0<e^{-2c\tau U}$ by  our selection, we have
\begin{eqnarray*}
{\tilde\epsilon_0}^{({4\over3})^N}&=&{\tilde\epsilon_0}e^{({({4\over3})^N}-1)\ln{\tilde\epsilon_0}}\leq\tilde\epsilon_{0}e^{\ln{\tilde\epsilon_0Q_{n+1}^{2^{n+1}c\tau U\ln{ 4\over3}}}}\\
&&\leq\tilde\epsilon_{0}e^{-{2^nc\tau U}Q_{n+1}^{2^{n+1}c\tau U\ln{ 4\over3}}}\leq \tilde\epsilon_{0}Q_{n+2}^{-2^{n+2}c\tau U}= \epsilon_{n+1},
\end{eqnarray*}
therefore $$|X_{P_{n+1}}|^{*}_{s_{n+1},D_{n+1}\times \mathcal O_{n+1}}\leq \tilde\epsilon_N \leq  \epsilon_{n+1}.$$

The estimation below is a direct calculous,
$$|B_p^{n+1}|^*_{r_{n+1},{\mathcal  O_{n+1}}} \leq {\mathcal E}_n+\tilde{\mathcal E}_N\leq {\mathcal E}_{n}+\sum\limits_{i=0}^{j-1}\tilde\epsilon_i\leq {\mathcal E}_{n}+2\epsilon_n\leq{\mathcal E}_{n+1}. $$
Similarly, one has  $|e_{n+1}|^*_{r_{n+1},{\mathcal  O_{n+1}}} \leq {\mathcal E}_{n+1}$ and $|{\it\Omega}_p^{n+1}|^*_{\mathcal  O_{n+1}}\leq {\mathcal E}_{n+1}$.

To prove $(\ref{phi-1-1})$ and  $(\ref{phi-1-2})$, for any $0\leq \nu\leq N-1$, we let $$\phi^\nu=\phi_{F_0}\circ\phi_{F_1}\circ\ldots\circ\phi_{F_{\nu}}: \tilde D_{\nu+1}\times \tilde {\mathcal O}_{\nu}\rightarrow \tilde D_0.$$
  By the chain rule,  we have estimate,
\begin{eqnarray*}\pmb{\pmb |}D\phi^\nu\pmb{\pmb |}_{\tilde s_0,\tilde s_{\nu+1},\tilde D_{\nu+1}}\leq \prod\limits_{\mu=0}^{\nu} \pmb{\pmb |}D\phi_{F_{\mu}}\pmb{\pmb |}_{\tilde s_{\mu+1},\tilde s_{\mu+1},\tilde D_{\mu+1}}\leq e^{ \sum\limits_{\mu=0}^{\nu} \tilde\epsilon_{\mu}^{3\over4}}\leq e,\end{eqnarray*}
\begin{eqnarray*}\pmb{\pmb |}D\phi^\nu \pmb{\pmb |}^{\mathcal L}_{\tilde s_0,\tilde s_{\nu+1},\tilde D_{\nu+1}}&\leq&\sum\limits_{\mu=0}^{\nu}\pmb{\pmb |}D\phi_{F_\mu}\pmb{\pmb |}_{\tilde s_{\mu},\tilde s_{\mu+1},D_{\mu+1}}^{\mathcal L}\prod\limits_{j=0,j\neq\mu}^{\nu} \pmb{\pmb |}D\phi_{F_j}\pmb{\pmb |}_{\tilde s_{j+1},\tilde s_{j+1},\tilde D_{j+1}}\nonumber \\&\leq&  e\sum\limits_{\mu=0}^{\nu}\pmb{\pmb |}D\phi_{F_\mu}\pmb{\pmb |}^{\mathcal L}_{\tilde s_{\mu},\tilde s_{\mu+1},D_{\mu+1}}\leq 2e\sum\limits_{\mu=0}^{\nu}\tilde\epsilon^{3\over4}_\nu\leq c\epsilon_n^{3\over4}.\end{eqnarray*}
With mean value theorem,
\begin{eqnarray*}&&|\phi^\nu-id|^*_{\tilde s_0,{\tilde D_{\nu+1}}}\\&\leq &\sum\limits_{\mu=0}^{\nu-1}|\phi^{\mu+1}-\phi^\mu|^*_{\tilde s_0,\tilde D_{\mu+1}}\leq \sum\limits_{\mu=0}^{\nu-1}\pmb{\pmb |}D\phi^{\mu}\pmb{\pmb |}^*_{\tilde s_0,\tilde s_{\mu+1}}|\phi_{F_{\mu+1}}-id|^*_{\tilde s_{\mu+1},\tilde D_{\mu+1}}\\
&\leq&2\sum\limits_{\mu=0}^{\nu-1}\tilde\epsilon_{\mu}^{3\over4}\leq 2\sum\limits_{\mu=0}^{N-1}\tilde\epsilon_{\mu}^{3\over4}  \leq c\epsilon_n^{3\over4},\nonumber\end{eqnarray*}
and then with generalized Cauchy estimate,
\begin{eqnarray*}&&\pmb{\pmb |}D\phi^\nu-I\pmb{\pmb |}^*_{\tilde s_0,\tilde s_{\nu+1},\tilde D_{\nu+1}}
\leq {c\epsilon_n^{3\over4}*Q_{n+1}^3\over r_0}\leq \epsilon_n^{2\over3}.\end{eqnarray*}
%\begin{eqnarray*}&&|D\Phi^\nu_{n+1}-I|_{\tilde s_0,\tilde s_{\nu},\tilde D_\nu}\leq \prod\limits_{\mu=1}^\nu |D(\Phi^\mu_{n+1}-\Phi^{\mu-1}_{n+1})|_{\tilde s_{\mu},\tilde s_{\mu},\tilde D_\mu}\\
%&&\leq\prod\limits_{\mu=1}^\nu |D^2\phi_{n+1}^\mu|\cdot |D\Phi^{\mu-1}_{n+1}|_{\tilde s_{\mu},\tilde s_{\mu},\tilde D_\mu}\leq \sum_{\mu=1}^\nu (1+\tilde\epsilon_{\mu-1}^{3\over4})\tilde\epsilon_{\mu}^{3\over4}\leq c\tilde\epsilon_0^{2\over3}\leq \epsilon_n^{2\over3}.\end{eqnarray*}
Therefore $\Phi_{n+1}=\phi^{N-1}$ is the transformation we are searching.\end{proof}\qed

%Since $\Omega_p^{n+1}=\tilde\Omega_p^{N}$ and $\mathcal O_{n+1}=\tilde {\mathcal O}_N$ by  \eqref{setting},
 % For $\xi\in \mathcal O_{n+1}$ and $|k|+|l|\leq$,

\subsection{Measure estimate}

At the ${j-1}$-th finite KAM iteration of the $n$-th infinite iteration,  we have to  exclude the following resonant set:
\begin{equation}\label{resonantset}\tilde{\mathcal
R}^{ j}=\bigcup_{|k|+|l|\leq \tilde K_j}(\bigcup_{ p\geq 1} \mathcal  R_{klp}^{ j}\bigcup_{|l|+|p-q|\neq0} \mathcal  R_{klpq}^{ j11
}\bigcup_{ p,q\geq 1}\mathcal  R_{klpq}^{ j2
}),\end{equation}
where
\begin{equation*}
\begin{array} {l}
\mathcal  R_{klp}^{j}=\{\xi\in \tilde{\mathcal
O}_{j-1}:|\langle k,\omega_1\rangle+\la l,\omega_2\ra \pm\tilde{\mathbf\Omega}_{p}^{j}|< {\tilde\gamma_j \over (|k|+|l|+1)^{\tau}},\},\\
 \mathcal  R_{klpq}^{j11}=\{\xi\in \tilde{\mathcal
O}_{j-1}:|\langle k,\omega_1\rangle+\la l,\omega_2\ra+\tilde{\mathbf\Omega}_{p}^{j}-\tilde{\mathbf\Omega}_q^{j}| < {\tilde\gamma_j \over (|k|+|l|+1)^{\tau}}\},\\
\mathcal  R_{klpq}^{j2}=\{\xi\in\tilde{\mathcal
O}_{j-1}:|\langle k,\omega_1\rangle+\la l,\omega_2\ra\pm(\tilde{\mathbf\Omega}_{p}^{j}+\tilde{\mathbf\Omega}_q^{j})| < {\tilde\gamma_j \over (|k|+|l|+1)^{\tau}}\}.
\end{array}
\end{equation*}

In order to estimate the measure of the resonant set $ \tilde{\mathcal
R}^{ j}$, we first need the following observation:

\begin{Lemma}\label{emptyset1}
For any  $|k|+|l|\leq \tilde K_{j-1}$ and $p,q\geq 1$, then the resonant set satisfy $$\mathcal  R_{klp}^{j}=\mathcal  R_{klpq}^{j11}=\mathcal  R_{klpq}^{j2}=\emptyset.$$
\end{Lemma}
\begin{proof}
As an example, we prove that $$\mathcal  R_{klpq}^{j11}=\emptyset,\, \quad if \quad  |k|+|l|\leq \tilde K_{j-1}.$$
%we will prove that the diophantine condition are automatically satisfies for  and $\xi\in \tilde{\mathcal O}_{j-1}$.
By the regularity of $X_{\tilde P_{j-1}}$, one has $$|{\tilde{\mathbf\Omega}}_p^j-{\tilde{\mathbf\Omega}}_p^{j-1}|_{\tilde{\mathcal O}_{j-1}}\leq \tilde\epsilon_{j-1}, p\geq 1.$$
It follows that for any $|k|+|l|\leq \tilde K_{j-1}$, there is
\begin{eqnarray*}&&|\langle k,\omega_1\rangle+\la l,\omega_2\ra+{\tilde{\mathbf\Omega}}_p^j-{\tilde{\mathbf\Omega}}_q^j|\\
&\geq& |\langle k,\omega_1\rangle+\la l,\omega_2\ra+{\tilde{\mathbf\Omega}}_p^{j-1}-{\tilde{\mathbf\Omega}}_q^{j-1}|-|\tilde{\mathbf\Omega}_p^{j}-\tilde{\mathbf\Omega}_p^{j-1}|-|\tilde{\mathbf\Omega}_q^{j}-\tilde{\mathbf\Omega}_q^{j-1}|\\
&\geq& {\tilde\gamma_{j-1}\over (|k|+|l|+1)^\tau}-2\tilde\epsilon_{j-1}\geq {\tilde\gamma_j\over (|k|+|l|+1)^\tau}. \end{eqnarray*}
The last inequality is possible since
 $\tilde\epsilon_{j-1} |\tilde K_{j-1}|^\tau\leq (\tilde\gamma_{j-1}-\tilde\gamma_j)$ by the iteration sequence \eqref{minor iteration sequence}. \qed
\end{proof}

\begin{Lemma}\label{emptyset2}
  If $ \max\{p,q\}> c(|k|+|l|), p\neq q$, then the resonant set satisfy $$\mathcal  R_{klp}^{j}=\mathcal  R_{klpq}^{j11}=\mathcal  R_{klpq}^{j2}=\emptyset.$$

\end{Lemma}
\begin{proof}
As an example,  we only prove that $$\mathcal  R_{klpq}^{j11}=\emptyset,\,  \quad if \quad  \max\{p,q\}> c(|k|+|l|), p\neq q.$$
In fact, it follows from the following computations:
 \begin{eqnarray*}
&&|\langle k,\omega_1\rangle+\la l,\omega_2\ra +\tilde{\mathbf\Omega}_{p}^{j}-\tilde{\mathbf\Omega}_q^{j}|\\
&\geq&  |p^2-q^2|-|\tilde{\mathbf\Omega}_{p}^{j}-p^2|-|\tilde{\mathbf\Omega}_q^{j}-q^2|-|\xi|(|k|+|l|)|\bar\omega|\\
&\geq& |p^2-q^2|-2\epsilon_0-{3\over2}|\bar\omega|(|k|+|l|)>{|p^2-q^2|\over2}.
\end{eqnarray*} The  others can be handled in the same way. \qed
\end{proof}

\begin{Lemma}\label{emptyset3}
If  $l\neq 0$ and $p=q$, then the resonant set satisfy $\mathcal  R_{klpq}^{j11}=\emptyset$.\
\end{Lemma}
\begin{proof}
In this case, since we assume $\bar{\omega}\in WL(\gamma,\tau,\beta),$ then
 \begin{eqnarray*}
 &&|\langle k,\omega_1\rangle+\langle l,\omega_2\rangle +{\tilde{\mathbf\Omega}}_p^{j}-{\tilde{\mathbf\Omega}}_q^{j}|\\
 &=&\xi|\langle k,\bar\omega_1\rangle+\langle l,\bar\omega_2\rangle |\geq {\gamma_0\over 2 (|k|+|l|+1)^{\tau}}\geq {\tilde\gamma_j\over 2 (|k|+|l|+1)^{\tau}}\end{eqnarray*}
 which just means  $\mathcal  R_{klpq}^{j11}=\emptyset$.
 \qed
 \end{proof}

By Lemma \ref{emptyset1}, \ref{emptyset2} and \ref{emptyset3}, one can reduce the resonant set \eqref{resonantset}  to be following set:

\begin{equation}\label{resonant2}
\begin{array}{lll}\tilde{\mathcal
R}^{ j}%&=&\bigcup\limits_{0<|k|\leq K_j}(\bigcup\limits_{ p\in\N} \mathcal  R_{kp}^{ j}\bigcup\limits_{p\neq q\in \N} \mathcal  R_{kpq}^{j11}\bigcup\limits_{p\neq q\in \N}\mathcal  R_{kpq}^{ j02}\bigcup\limits_{p\neq q\in \N}\mathcal  R_{kpq}^{ j20})
=\bigcup\limits_{\tilde K_{j-1}<|k|+|l|\leq \tilde K_j}(\bigcup\limits_{  p\leq c(|k|+|l|)} \mathcal  R_{klp}^{ j}\bigcup\limits_{|l|+|p-q|\neq 0,\atop  p,q\leq c(|k|+|l|)} \mathcal  R_{klpq}^{ j11
}\bigcup\limits_{  p,q\leq c(|k|+|l|)}\mathcal  R_{klpq}^{ j2
}).
\end{array}
\end{equation}

\begin{Lemma} \label{finite total measure}
At $n$-th infinite iteration, we have to exclude the following  resonant set $\mathcal O_n\backslash \mathcal O_{n+1}= \bigcup\limits_{j=1}^N\tilde{\mathcal R}^j$,  and there exist constant $C$ such that
\begin{eqnarray}\label{small measure} meas(\tilde{\mathcal R}^j) &\leq& {C\tilde\gamma_j\over \tilde K_{j}^2},\\
\label{small measure-1} meas(\mathcal O_{n}\backslash \mathcal O_{n+1}) &\leq&{C\gamma_n\over K_n}.\end{eqnarray}
\end{Lemma}
\begin{proof}
 By the definition,  we have $\mathcal O_{n}=\tilde{\mathcal O}_{0}$. By Lemma \ref{ite},   one has $\mathcal O_{n+1}=\tilde{\mathcal O}_{N}$. By \eqref{minor iteration sequence}, one has  $\mathcal O_{n}\backslash \mathcal O_{n+1}=\tilde{\mathcal O}_{0}\backslash\tilde{ \mathcal O}_N= \bigcup\limits_{j=1}^N\tilde{\mathcal R}^j$ at once.

Now, as an example, we will  focus on the measure of resonant set $$\mathcal  R_{klpq}^{j11}=\{\xi\in \tilde{\mathcal
O}_{j-1}:|\langle k,\omega_1\rangle+\la l,\omega_2\ra+\tilde{\mathbf\Omega}_{p}^{j}-\tilde{\mathbf\Omega}_q^{j}| < {\tilde\gamma_j \over (|k|+|l|+1)^{\tau}}\}$$ with  $p\neq q$ and $(k,l)\neq0$.
Since $\xi\in\tilde{\mathcal O}_{j-1}\subseteq ({1\over2},{3\over2})$, it follows that
 $$\mathcal  R_{klpq}^{j11}\subseteq Q_{klpq}^{j11}=\{\xi\in \tilde{\mathcal
O}_{j-1}:|\langle k,\bar\omega_1\rangle+\la l,\bar\omega_2\ra+{\tilde{\mathbf\Omega}_{p}^{j}-\tilde{\mathbf\Omega}_q^{j}\over\xi}| < {2\tilde\gamma_j \over (|k|+|l|+1)^{\tau}}\}$$

By a direct computation,
$$|{d\over d\xi} (\langle k,\bar\omega_1\rangle+\la l,\bar\omega_2\ra+ {1\over\xi}(\tilde{\mathbf\Omega}_{p}^{j}-\tilde{\mathbf\Omega}_q^{j}))|\geq {1\over 9}|p^2\pm q^2|,$$
then
$$meas(\mathcal  R_{klpq}^{ j11})\leq meas(Q_{klpq}^{ j11})\leq {2\tilde\gamma_j \over 9(|k|+|l|+1)^{\tau}}{1\over |p^2-q^2|}.$$
%Similarly, we have
%\begin{eqnarray*}
%&meas(\mathcal  R_{klp}^{ j})\leq {\tilde\gamma_j \over 36(|k|+|l|)^{\tau}}{1\over p^2},\\
%&meas(\mathcal  R_{klpq}^{ j20})\leq {\tilde\gamma_j \over 36(|k|+|l|)^{\tau}}{1\over p^2+q^2},
%&meas(\mathcal  R_{klpq}^{ j02})\leq {\tilde\gamma_j\over 36(|k|+|l|)^{\tau}}{1\over p^2+q^2}.
%\end{eqnarray*}
Thus the total measure can be estimates as: \begin{eqnarray*}
%&\leq&\sum \limits _{K_{j}<|k|\leq K_{j+1}}(\sum \limits_{  1\leq  p\leq c(\alpha)|k|} meas(\mathcal  R_{kp}^{ j})+\sum\limits_{p\neq q,\atop 1\leq p,q\leq c(\alpha)|k|} meas(\mathcal  R_{kpq}^{ j11})\\&&+2\sum \limits_{ 1\leq p,q\leq c(\alpha)|k|}meas(\mathcal  R_{kpq}^{ j2}))\\
meas(\tilde{\mathcal R}^j)\leq&\sum \limits _{K_{j}<|k|+|l|\leq K_{j+1}\atop (k,l)\in\Z^2\times\Z^d}&(\sum \limits _{ 1\leq p\leq c(|k|+|l|)}{2\tilde\gamma_j \over 9(|k|+|l|+1)^{\tau}}{1\over p^2} \\
&&+\sum \limits _{p\neq q,\atop1\leq p,q\leq c(|k|+|l|)}{2\tilde\gamma_j \over  9(|k|+|l|+1)^{\tau}}{1\over| p^2-q^2|}\\&&
+\sum \limits _{  1\leq p,q\leq c(|k|+|l|)}{4\tilde\gamma_j \over 9(|k|+|l|+1)^{\tau}}{1\over p^2})\\
\leq&\sum \limits _{K_{j}<|k|+|l|\leq K_{j+1}}&\big({4\tilde\gamma_j \over (|k|+|l|+1)^{\tau}}+{2\tilde\gamma_j \ln(|k|+|l|)\over (|k|+|l|+1)^{\tau-2}}\\&&+{4\tilde\gamma_j \over (|k|+|l|+1)^{\tau-2}}\big)\\
\leq& {C\tilde\gamma_j\over K_{j}^2}.&
\end{eqnarray*}
The last inequality is possible since we choose $\tau>d+6$. The measure of the other resonant set be estimates similarly. There is also
 $$meas(\mathcal O_n\backslash \mathcal O_{n+1}) \leq \sum_{i=1}^N {C\tilde\gamma_j\over K_{j}^2}\leq {C\tilde\gamma_1\over \tilde K_0}\leq {C\gamma_n\over K_n}.$$
  \qed
\end{proof}

\appendix\section{Appendix: Proof of Lemma \ref{Lem4.2}}
\begin{proof}
At the $j$-th step of the finite iteration, the Hamiltonian $\tilde H_j = \tilde{\mathcal N}_j +\tilde P_j$
is studied as a small perturbation of some normal form $\tilde{\mathcal N}_j$. A transformation $\phi_{F_{j}}^1$ is set up so  that
$\tilde H_j\circ\phi_{F_{j}}^1 =\tilde{\mathcal N}_{j+1} +\tilde P_{j+1}$
with new normal form $\tilde{\mathcal N}_{j+1}$ and a much smaller perturbation $\tilde P_{j+1}$.
We drop the index $j$ of $\tilde H_j ,
\mathcal N_j , \tilde P_j , \phi_{F_j}^1$ and shorten the index $j+1$ to be $+$.

 Let $R$ to be $2$-order Taylor polynomial truncation of $\tilde P$, that is
 \begin{eqnarray}\label{truncation}
 R&=&\sum\limits_{|k|+|l|\leq\tilde K,\atop |\alpha|+|\beta|\leq 2} {\tilde P}_{kl\alpha\beta}e^{i\la(k,l),(\theta,\varphi)\ra}z^\alpha\bar z^\beta\\
 &=:&R^0+\la R^{01},z\ra+\la R^{10},z\ra+\la R^{02}z,z\ra+\la R^{ 11} z,\bar z\ra+\la R^{20}\bar z,\bar z\ra,\nonumber
 \end{eqnarray}
 where $\la \cdot,\cdot\ra$ is formal product for two column vectors,  $R^0, R^{01}, R^{10}, R^{02},$ $ R^{11}$ and $ R^{20}$
  depend on $\theta,\varphi$ and $\xi$.

  By $\llbracket R\rrbracket$ denote the part of $R$ in generalized average part as follows
$$\llbracket R\rrbracket=[R^0]_{\varphi}+\la [diag(R^{ 11})]_{\varphi} z,\bar z\ra,$$
where $diag(R^{11})$ is the the diagonal of $R^{11}$.
The transformation $\phi_F^1=X^1_F$ is constructed as the time-$1$-map
 of a Hamiltonian vector field $X_F$, where $F$ is of the same form as $R$,
  $$F=F^0+\la F^{01},z\ra+\la F^{10},z\ra+\la F^{02}z,z\ra+\la F^{ 11} z,\bar z\ra+\la F^{20}\bar z,\bar z\ra,$$
and  $\llbracket F(\theta,\varphi)\rrbracket=0$. The function $F(\theta,\varphi)$  is also an approximate solution of the homological equation
 %, we will construct  a symplectic vector field $\Phi=:X_{F}$, where $f$ is the solution of homological equation
\begin{equation}\label{homo}
\{\tilde{\mathcal N},F\}=R-\sum\limits_{ p\geq1}[{R^{11}_{pp}}(\theta,\varphi)]_{\varphi}|z_p|^2-\llbracket R\rrbracket\end{equation}
with $\tilde{\mathcal N}={\tilde e}(\theta;\xi)+\la \omega_1,I\ra+\la \omega_2,J\ra+\la( \tilde{\mathbf \Omega}+ B(\theta;\xi)+b(\theta;\xi)) z,\bar z\ra$.
%Instead of solve this equation directly, we consider solution of problem
 %Since $\tilde B_n(\theta;\xi)$ depend on $\theta$, homological equation \eqref{homo} may have no solution,
%\begin{equation}\label{approximatehomolequation}\{{ N},F\}=R-\sum\limits_{ p\geq1}{R^{11}_{pp}}(\theta)|z_p|^2.\end{equation}
In the following, we denote $\partial_{(\omega_1,\omega_2)}=\la\omega_1,\partial_\theta\ra+\la\omega_2,\partial_\varphi\ra$. Then $F$ should satisfy the homological equations:
\begin{equation}\label{homo-13}
\begin{array}{ll}
\partial_{(\omega_1,\omega_2)}F^{0}=R^{0}-[R^0]_{\varphi},&\\
\partial_{(\omega_1,\omega_2)}F^{01}_{p}-i( {\tilde{\mathbf\Omega}_p}+ B_p+ b_p)
F^{01}_p=R^{01}_{p},&p\geq 1,\\
\partial_{(\omega_1,\omega_2)}F^{10}_{q}+i( {\tilde{\mathbf\Omega}_q}+B_q+ b_q)
F^{10}_q=R^{10}_{q},&q\geq 1,\\
\partial_{(\omega_1,\omega_2)}F^{20}_{pq}-i( {\tilde{\mathbf\Omega}_p}+ B_p+b_p+{\tilde{\mathbf\Omega}_q}+ B_q+ b_q)
F^{02}_{pq}=R^{z z}_{02},&p,q\geq 1,\\
\partial_{(\omega_1,\omega_2)}F^{20}_{pq}+i( {\tilde{\mathbf\Omega}_p}+ B_p+ b_p+{\tilde{\mathbf\Omega}_q}+ B_q+ b_q)
F^{20}_{pq}=R^{20}_{pq},&p,q\geq 1,\\
\partial_{(\omega_1,\omega_2)}F^{11}_{pq}+i( {\tilde{\mathbf\Omega}_p}+B_p+ b_p-{\tilde{\mathbf\Omega}_q}-B_q- b_q)
F^{ 11}_{pq}=R^{11}_{pq},&p\neq q,\\
\partial_{(\omega_1,\omega_2)}F^{11}_{pq}=R^{11}_{pq}-[R_{pq}^{11}]_{\varphi},&p=q.
\end{array}
\end{equation}

For the first equation from  \eqref{homo-13},
%The first and the last equation in \eqref{homo-13} is solved in the same way. For example, we consider equation
 $$\la\partial_\theta F^{0},\omega_1\ra+\la\partial_\varphi F^{0},\omega_2\ra=R^{0}-[R^0(\theta,\varphi)]_{\varphi}.$$
 The fourier expansion is given, and we arrive at equation $$i(\la k,\omega_1\ra+\la l,\omega_2\ra){ \hat F}^{0}_{(k,l)}={\hat R}^0_{(k,l)} $$
for  $(k,l)\in \Z^2\times \Z^d$ with  $|k|+|l|\leq K$ and $l\neq 0$. Recall $\omega=\xi\bar\omega$ and by weak Liouvillean condition
 \eqref{melnikov}, we have
$$|F^0(\theta,\varphi)|^*_{\tilde r-3\tilde\sigma}\leq {{\tilde \gamma}^{-1}\tilde \sigma}^{-3-\tau}   |R^{0}|^*_{\tilde r}.$$
The last equation in  \eqref{homo-13} is considered in the same way and  $$|F^{11}_{pp}(\theta,\varphi)|^*_{\tilde r-3\tilde\sigma}\leq {{\tilde \gamma}^{-1}\tilde \sigma}^{-3-\tau}   |R_{pp}^{11}|^*_{\tilde r}.$$

The left equations in \eqref{homo-13} will be discussed in the same way, as an example,  we do this for
 \begin{equation}\label{homo16}\partial_{(\omega_1,\omega_2)}F^{11}_{pq}+i({\tilde{\mathbf\Omega}_p}+B_p+ b_p-{\tilde{\mathbf\Omega}_q}-B_q- b_q)
F^{ 11}_{pq}=R^{11}_{pq}, p\neq q.\end{equation}
To obtain a solution of these equations with useful estimates we want to apply Proposition  \ref{basic lemma}. The assumptions of this proposition are now verified.

We  set $(\mathbbm{r}, \tilde{\mathbbm{r}},\tilde\sigma,\gamma, K,({\gamma\over 2\eta})^{1\over \tau+2},\eta_1,\eta_2,\tilde\eta)$ to be  $(r_n,\tilde r,\tilde\sigma,\tilde \gamma,\tilde K, K^n,\mathcal E_n,\tilde{\mathcal E},\tilde\epsilon)$. The assumption $[B_{p}(\theta;\xi)]=[b_{p}(\theta;\xi)]=0$, $|B_p|_{r_n}\leq \mathcal E_n$, $|b_p|_{\tilde r}\leq \tilde {\mathcal E}$ and $|R^{11}_{pq}|_{\tilde r}\leq\tilde\epsilon\leq \tilde {\mathcal E} $ are given by condition from Lemma \ref{Lem4.2}.

\noindent {\it Verification of assumption \ref{cond1}:} In fact, by Lemma \ref{sup-liouvillean}, one has $Q_{n+1}\geq Q_n^{\mathcal A}$ and $\ln Q_{n+1}\leq Q_n^U$ for any $n\geq 1$.   Then since $r\leq 2r_{n+1}<r_n$, one has
\begin{eqnarray*}e^{-|r_n-r|Q_{n+1}}&\leq& e^{-{Q_{n+1}\over 8Q_n^3}}\leq e^{-{Q_{n+1}^{1\over2}\ln Q_{n+1}\over 8Q_n^3}}
\leq e^{-\ln Q_{n+1}Q_n^{{\mathcal A\over2}-3}}\\
&\leq&  e^{-\ln Q_{n+1} n2^{n+1}c\tau U} =Q_{n+1}^{-n2^{n+1}c\tau U}\leq\epsilon_n\leq\tilde {\mathcal E}.\end{eqnarray*}
Thus we have our conclusion as $ \mathcal E_n<1$.

\noindent {\it Verification of assumption \ref{cond2}:} This is a direct computation since $\mathbbm{r}_1=r_n$, $\mathbbm{r}_2=\tilde r\leq r_{n+1}$ and $|B_p|_{r_n}\leq \mathcal E_n$.

\noindent {\it Verification of assumption \ref{cond3}:}
Since $K^n=[({\gamma_0\over 2\epsilon_n})^{1\over \tau+2}]$ and by \eqref{minor iteration sequence}, one has
\begin{eqnarray}\label{key inequality}
\tilde K&\leq&\ln ({ 1\over\tilde\epsilon_0})^{({4\over3})^N}({2^NQ^3_{n+1}
\over r_0})=  {1\over r_0} {({8\over3})^N}Q_{n+1}^3\ln\tilde\epsilon_0^{-1}
   \\&=& {1\over r_0} {({8\over3})^{[{2^{n+1}c{\tau}U\ln Q_{n+1}\over 6\tau+9}]+2}}Q_{n+1}^3\ln\tilde\epsilon_0^{-1}
   \leq Q_{n+1}^{({2^{n+1}c{\tau}U\over 6\tau+9}+5)}\ln\tilde\epsilon_0^{-1}\nonumber\\
   &\leq& Q_{n+1}^{2^{n+1}c \tau U\over 6\tau+9}\cdot Q_{n+1}^3\cdot \tilde\epsilon_0^{-c(\tau)}\nonumber\\
   &\leq& {\tilde\epsilon_n}^{-{1\over 6\tau+9}}\cdot\tilde\epsilon_0^{-{1\over (12\tau+18)}}\cdot \tilde\epsilon_0^{-{1\over (12\tau+18)}} \nonumber\\
   &\leq&{\tilde\epsilon_0}^{-{2\over 6\tau+9}}\leq K^n.\nonumber
\end{eqnarray}

\noindent {\it Verification of assumption \ref{cond4}:}
For any $\xi\in\tilde{\mathcal O}$, we consider any pair $((k,l),p,q)$ with $|k|+|l|\leq {\tilde K}$ and $p\neq q$:

 If $\max\{ p,q\}\geq c(|k|+|l|)$, one has
$$|\la k,\omega_1\ra+\langle l,\omega_2\ra+\tilde{\mathbf\Omega}_{p}-\tilde{\mathbf\Omega}_q|\geq |p^2-q^2|-|\tilde{\mathbf\Omega}_{p}-p^2|-|\tilde{\mathbf\Omega}_q-q^2|-(|k|+|l|)|\omega|\geq {|p^2-q^2|\over2}.$$

If $\max\{ p,q\}\leq (1+\alpha)(|k|+|l|)$, by \eqref{melnikov}, one has
$$|\la k,\omega_1\ra+\langle l,\omega_2\ra +\tilde{\mathbf\Omega}_{p}-\tilde{\mathbf\Omega}_q|\geq {\tilde\gamma\over (|k|+|l|+1)^{\tau} } \geq {\tilde\gamma\over (|k|+|l|+1)^{\tau+2} } |p^2-q^2|.$$
Recall $\omega=\xi\bar\omega$, one has %$\forall\xi\in \tilde{\mathcal O}$, if $ |k|+|l|\leq {\tilde K}$ and $p,q\geq1$ with $p\neq q$, there is
\begin{equation}\label{lowerbound}
|\la k,\bar\omega_1\ra+\langle l,\bar\omega_2\ra +{\tilde{\mathbf\Omega}_p-\tilde{\mathbf\Omega}_q\over\xi}|\geq {\tilde\gamma\over (|k|+|l|+1)^{\tau+2} }|p^2-q^2|.
\end{equation} We have our assumption once $\zeta$ and $\lambda$ are in placed by $\tilde{\mathbf\Omega}_{p}-\tilde{\mathbf\Omega}_q\over\xi$ and $p^2-q^2$.

Thus, Proposition  \ref{basic lemma}  applies, and the approximate solution $F^{11}_{pq}$ satisfies the estimate
\begin{eqnarray}\label{operatornorm}
|F^{11}_{pq}|_{D(\tilde r-2\tilde\sigma)}&\leq&{{\tilde \gamma}^{-1}\tilde \sigma}^{-3-\tau} |p^2-q^2|^{-1}  |R^{11}_{pq}|_{D( \tilde r-\tilde\sigma)}.
\end{eqnarray}
To obtain the norm of $X_F$, we  need following useful Lemma.
\begin{Lemma}\label{matrixnorm}{(\textbf{M.3 in \cite{KPo}})} Let $R = (R_{pq})_{p, q\geq 1}$ be a bounded operator on $\ell^2$ which depends on $x\in \T^n$ such that all
elements $(F_{pq})$ are analytic on $D(r)$. Suppose $F= (F_{pq})_{ p,q\geq 1}$ is another operator on $\ell ^2$ depending on
$x$ whose elements satisfy
\[\sup_{x\in D(r)}|F_{pq}(x)|\leq {1\over|p- q|}\sup_{x\in D(r)}
|R_{pq}(x)|, p \neq q, \]
and $F_{pq} = 0$. Then $F$ is a bounded operator on $\ell^2$ for every $x \in D(r)$, and
\[\sup_{x\in D(r-\sigma)}\|F(x)\|\leq {c\over\sigma^2}\sup_{x\in D(r)}
\|R(x)\|.\]
\end{Lemma}
Then recall \eqref{operatornorm}, we have
\begin{equation*}
|F^{11}|_{a,\rho,D{(\tilde r-3\tilde\sigma)}}\leq {256 \over{{\tilde \gamma}\sigma}^{5+\tau}}  |R^{11}|_{a,\rho,D{(\tilde r-\tilde\sigma)}}\leq {256 \over{{\tilde \gamma}\sigma}^{5+\tau}} |X_R|_{D(\tilde r,\tilde s)}.
 \end{equation*}
Multiplying by $z,\bar z$ we get

\[{1\over \tilde s^2}|\la F^{11}z,\bar z\ra|_{D(\tilde r-3\tilde \sigma,\tilde s)}\leq  |F^{11}|_{a,\rho,D{(\tilde r-3\tilde\sigma,\tilde s)}},\]
and finally by Cauchy's estimate we have

$$|X_{\la F^{11}z,\bar z\ra}|_{\tilde s,D(\tilde r-3\tilde\sigma,\tilde s)}\leq  {256\over \tilde\gamma\tilde\sigma^{6+\tau}}|X_R|_{\tilde s,D(\tilde r,\tilde s)}.$$

To obtain the estimate of the Lipschitz semi-norm, we proceed as follows. Shortening $\triangle_{\xi\eta}$ to be $\triangle$ and applying it to $(\ref{homo16})$, one gets that
\begin{eqnarray}\label{rotation homological equation2}
&&\partial_{(\omega_1,\omega_2)}\triangle F_{pq}^{11}+i\mathcal T_K(( {\tilde{\mathbf\Omega}_p}+ B_p+b_p-{\tilde{\mathbf\Omega}_q}- B_q-b_p)\triangle F_{pq}^{11})\\
&=&-i\mathcal T_K(\triangle ( {\tilde{\mathbf\Omega}_p}+ B_p+b_p-{\tilde{\mathbf\Omega}_q}-B_q-b_p)
F_{pq}^{11})+i\triangle R_{pq}^{11}:=Q_{pq}\nonumber
\end{eqnarray}
By \eqref{lowerbound} and \eqref{operatornorm},
\begin{eqnarray*}|Q_{pq}|_{D(\tilde r-2\tilde\sigma)}&\leq& |\triangle( {\tilde{\mathbf\Omega}_p}+B_p+b_p-{\tilde{\mathbf\Omega}_q}-B_q-b_q)|\cdot |F_{pq}^{11}|_{D(\tilde r-2\tilde\sigma)}+|\triangle  R^{11}_{pq}|_{D(\tilde r-2\tilde\sigma)}\nonumber  \\
&\leq &  {|\triangle( {\tilde{\mathbf\Omega}_p}+B_p+b_p-{\tilde{\mathbf\Omega}_q}-B_q-b_q)|\over|p^2-q^2|\tilde\gamma\tilde\sigma^{3}} |R_{pq}^{11}|_{D(\tilde r-\tilde\sigma)}+|\triangle  R^{11}_{pq}|_{D(\tilde r-\tilde\sigma)}.\end{eqnarray*}
 We again apply Proposition \ref{basic lemma}  to \eqref{rotation homological equation2}, one has
\begin{eqnarray*}&&|\triangle F_{pq}^{11}|_{D(\tilde r-2\tilde\sigma)}\\ &\leq&   {|\triangle( {\tilde{\mathbf\Omega}_p}+B_p+b_p-{\tilde{\mathbf\Omega}_q}- B_q-b_q)|\over|p^2-q^2|\tilde\gamma\tilde\sigma^{8+2\tau}} |R_{pq}^{11}|_{D(\tilde r-\tilde\sigma)}+|\triangle  R^{11}_{pq}|_{D(\tilde r-\tilde\sigma)},\end{eqnarray*}
Again by Lemma \ref{matrixnorm}, one has
\begin{eqnarray*}&&|\triangle F^{11}|_{a,\rho,D(\tilde r-3\tilde\sigma)} \\ &\leq&   {|\triangle( {\tilde{\mathbf\Omega}_p}+B_p+b_p-{\tilde{\mathbf\Omega}_q}- B_q-b_q)|\over|p^2-q^2|\tilde\gamma^2\tilde\sigma^{10+2\tau}}  |R^{11}|_{a,\rho,D(\tilde r-\tilde\sigma)}+|\triangle  R^{11}|_{a,\rho,D(\tilde r-\tilde\sigma)},\end{eqnarray*}
Dividing $|\triangle F_{pq}^{11}|_{D(\tilde r-3\tilde\sigma,\tilde s)}$ by $|\xi-\eta|\neq0$ and take supreme over $\tilde{\mathcal O}$, one gets
\begin{eqnarray*}
|X_{\la F^{11}z,\bar z\ra}|_{\tilde s,D(\tilde r-3\tilde\sigma,\tilde s)\times \tilde{\mathcal O}}^{\mathcal L}
%&\leq&  {|{\tilde\Omega_p-\tilde\Omega_q}|^{lip}+\|\tilde B_p-\tilde B_q\|^{\mathcal L}|\over|p^2-q^2|\tilde\gamma^2\tilde\sigma^{10+2\tau}}  \|R^{11}\|_{\rho,D(\tilde r-\tilde\sigma)}+ \|X_{\la R^{11}z,\bar z\ra}\|_{\tilde r,\tilde s}^{\mathcal L}\nonumber\\
&\leq& {|X_R|_{\tilde s,D(\tilde r-\tilde\sigma,\tilde s)\times \tilde{\mathcal O}}\over\tilde\gamma^2\tilde\sigma^{10+2\tau}} + |X_{\la R^{11}z,\bar z\ra}|_{\tilde s,D(\tilde r,\tilde s)\times \tilde{\mathcal O} }^{\mathcal L}
\end{eqnarray*}
Thus one has
\begin{equation}|\la X_{F_{pq}^{11}z,\bar z\ra}\ra|^*_{\tilde s,D(\tilde r-3\tilde\sigma,\tilde s)\times\tilde{\mathcal O}}\leq {|X_{\la R^{11}z,\bar z\ra}|_{\tilde s,D(\tilde r,\tilde s)\times\tilde{\mathcal O}}\over\tilde\gamma^2\tilde\sigma^{10+2\tau}}  + |X_{\la R^{11}z,\bar z\ra}|_{\tilde s,D(\tilde r,\tilde s)\times\tilde{\mathcal O}}^{\mathcal L}.\end{equation} For the left  terms of $F$, that is $\la F^{01},z\ra,\la F^{10},\bar z\ra, \la F^{02}z,z\ra$ and $\la F^{20}\bar z,\bar z\ra,$   one can obtain the same results with similar technical.

 %For estimation on $F^\theta$,  by a standard approach in finite dimensional KAM theory, we have
%\begin{equation}| X_{ F^{\theta} }|_{\tilde s,D(\tilde r-3\tilde\sigma,\tilde s)\times \tilde{\mathcal O}_+}^*\leq {| X_{ R }|_{\tilde s,D(\tilde r,\tilde s)\times \tilde{\mathcal O}}^*\over\tilde\gamma^2\tilde\sigma^{10+2\tau}}  \end{equation}
The estimation on  approximate solution $F$ is  obvious,
\begin{equation}| X_{ F }|_{\tilde s,D(\tilde r-3\tilde\sigma,\tilde s)\times \tilde{\mathcal O}}^*\leq {256\over\tilde\gamma^2\tilde\sigma^{10+2\tau}} | X_{ R }|_{\tilde s,D(\tilde r,\tilde s)\times \tilde{\mathcal O}}^*. \end{equation}
Let $\phi_F^1$ to be the time-$1$ map of $X_F^t$, we have
 \begin{eqnarray} \tilde H_{+}&=&\tilde H\circ\phi_{F}^1=(\mathcal N+\mathcal R)\circ X_F^1+(\tilde P-\mathcal R)\circ X_F^1\\
% &=&N+\{N,F\}+\mathcal R\\&&+\int_0^1 (1-t)\{\{{ N},F\},F\}\circ \phi_{F}^{t}dt+\int_0^1 \{\mathcal R,F\}\circ \phi_{F}^{t}dt +(\tilde P-\mathcal R)\circ\phi^1_{F}\nonumber \\
&=& \mathcal N+\{\mathcal N,F\}+\mathcal R\nonumber\\
&&+
 \int_0^1 \{\{(1-t){\mathcal N},F\}+\mathcal R,F\}\circ \phi_{F}^{t}dt +(\tilde P-\mathcal R)\circ
\phi^1_{F}\nonumber \\
&=&\mathcal N+\llbracket R\rrbracket+\mathcal R_e+
 \int_0^1 \{R(t),F\}\circ \phi_{F}^{t}dt +(\tilde P-\mathcal R)\circ
\phi^1_{F}\nonumber \\
&=&\mathcal N_++\tilde P_+,\nonumber
\end{eqnarray}
where  $\mathcal N_+=\mathcal N+\llbracket R\rrbracket$. For the new normal form $\mathcal N_+$, we set \[\tilde e_+=\tilde e+[R^0(\theta,\varphi)]_{\varphi},\,\tilde{\it\Omega}_p^+=\tilde{\it\Omega}_p+[R^{11}_{pp}(\theta,\varphi)],\,\tilde b_p^+= \tilde b_p+[R^{11}_{pp}(\theta)]_{\varphi}-[R^{11}_{pp}(\theta,\varphi)].\]
The perturbation
$$\tilde P_+=\mathcal R_e+ \int_0^1 \{R(t),F\}\circ \phi_{F}^{t}dt +(\tilde P-\mathcal R)\circ
\phi^1_{F}$$ with $R(t)=(1-t)\llbracket R\rrbracket+tR$ and
\[
\begin{array}{lll}
\mathcal R_e=&\sum\limits_{p\geq 1}e^{i\mathcal B_p}\Gamma_K(e^{-i\mathcal B_p}R_p-(b_p+\Gamma_K B_p)F^{01}_p)z_p\\
&+\sum\limits_{p\geq 1}e^{-i\mathcal B_p}\Gamma_K(e^{i\mathcal B_p}R_p+(b_p+\Gamma_K B_p)F_p^{10})\bar z_p\nonumber\\
&+\sum\limits_{p,q\geq 1,\atop  p\neq q}e^{i\mathcal B_p-i\mathcal B_p}\Gamma_K(e^{i\mathcal B_p-i\mathcal B_q}R_p-(b_p-b_q+\Gamma_K (B_p-B_q))F^{11}_{pq})z_p\bar z_q \nonumber\\
&+\sum\limits_{p,q\geq 1}e^{i\mathcal B_p+i\mathcal B_q}\Gamma_K(e^{-i\mathcal B_p-i\mathcal B_q}R_p-(b_p+b_q+\Gamma_K (B_p+B_q))F^{20}_{pq})z_p z_q\nonumber\\
&+\sum\limits_{p,q\geq 1}e^{-i\mathcal B_p-i\mathcal B_q}\Gamma_K(e^{i\mathcal B_p+i\mathcal B_q}R_p+(b_p+b_q+\Gamma_K (B_p+B_q))F^{02}_{pq})\bar z_p\bar z_q,\nonumber
\end{array}
\]
where operator $\Gamma_K=I-\mathcal T_K$.

Following \cite{P3}, one has $$|X_{\tilde P_+-\mathcal R_e}|^*_{\tilde s_+,D(\tilde r_+,\tilde s_+)}\leq \tilde\epsilon^{4\over3}.$$
 Recall \eqref{errorper},
\begin{equation}
|X_{\mathcal R_e}|^*_{\tilde s,D(\tilde r-5\tilde\sigma,\tilde s)}\leq e^{-K\tilde \sigma}\cdot |\sup_{p\geq 1}(|e^{|\texttt{Im} {\mathcal B}_p|_{\tilde r}^*}+|b_p|^*_{\tilde r})\cdot|X_F|^*_{\tilde s,D(\tilde r,\tilde s)}\leq \tilde\epsilon^{4\over3}.
\end{equation}
Thus we have new bound on perturbation $\tilde P_+$ and  Lemma \ref{Lem4.2} follows.\qed
\end{proof}
%\section{Auxiliary Result}
%For a bounded linear operator from $\ell^{a,p}$ to $\ell^{a,p}$, define its operator norm by $ |\cdot|_{a,p}$. The following lemma will be useful for us:
%\begin{Lemma}\label{matrixnorm}{(\textbf{M.3 in \cite{KPo}})} Let $R = (R_{pq})_{p, q\geq 1}$ be a bounded operator on $\ell^2$ which depends on $x\in \T^2$ such that all
%elements $(F_{pq})$ are analytic on $D(r)$. Suppose $F= (F_{pq})_{ p,q\geq 1}$ is another operator on $\ell ^2$ depending on
%$x$ whose elements satisfy
%\[\sup_{x\in D(r)}|F_{pq}(x)|\leq {1\over|p- q|}\sup_{x\in D(r)}
%|R_{pq}(x)|, p \neq q, \]
%and $F_{pq} = 0$. Then $F$ is a bounded operator on $\ell^2$ for every $x \in D(r)$, and
%\[\sup_{x\in D(r-\sigma)}\|F(x)\|\leq {c\over\sigma^2}\sup_{x\in D(r)}
%\|R(x)\|.\]
%\end{Lemma}
%
%\begin{Lemma}\label{Lemma600}(Lemma 7.8 in \cite {GY1}).
% Suppose that $g(u)$ is a $m^{th}$ differentiable function on the closure $\bar I$ of $I$, where
% $I\subset  \mathbb R$ is an interval. Let $I_{h}=\{u:|g(u)|<h\}, h>0$. If for some constant $d>0$, there is $|g^{(m)}(u)|>d$ for $\forall u\in I$, then $|I_n|<ch^{1\over m}$, where $|I_h|$ denotes the Lebesgue measure of $I_h$ and $c=2(2+3+\cdots+m+ d^{-1}).$
% \end{Lemma}

\section*{Acknowledgements}
X.Xu was supported by NSFC grant (11301072).  J. You  was partially supported by NSFC grant (11471155) and
973 projects of China (2014CB340701).   Q. Zhou was partially supported by \textquotedblleft Deng Feng Scholar Program B\textquotedblright of Nanjing University, Specially-appointed professor programme of Jiangsu province and NSFC grant (11671192).

\end{document}